\documentclass[11pt]{article}
\usepackage{amsmath,tabu}
\usepackage{amssymb}
\usepackage{amsfonts}
\usepackage{amsthm}
\usepackage{slashed}
\usepackage{cancel}
\usepackage{mathtools}
\usepackage{paralist}
\usepackage{graphicx}
\usepackage{float}
\usepackage{tikz}
\usepackage{pgfplots}
\usepackage[hmargin = 2.5 cm,vmargin = 2.7 cm]{geometry}
\usepackage{array}
\usepackage{hyperref}
\numberwithin{equation}{section}
\usepackage{cleveref}
\usepackage{cite}
\usepackage[affil-it]{authblk}
\usepackage{tikz-cd}
\usepackage{placeins}

\usepgfplotslibrary{fillbetween}
\usetikzlibrary{calc}

\newtheorem{lemma}[equation]{Lemma}
\newtheorem{proposition}[equation]{Proposition}
\newtheorem{theorem}[equation]{Theorem}
\newtheorem{corollary}[equation]{Corollary}
\newtheorem{thmx}{Theorem}

\theoremstyle{definition}
\newtheorem{definition}[equation]{Definition}
\newtheorem*{definition*}{Definition}

\theoremstyle{remark}
\newtheorem{remark}[equation]{Remark}
\newtheorem*{remark*}{Remark}

\newcommand{\Sp}{\mathrm{Sp}}
\newcommand{\SU}{\mathrm{SU}}
\newcommand{\U}{\mathrm{U}}

\newcommand{\SO}{\mathrm{SO}}
\newcommand{\CP}[1]{\mathbb{C}P^{#1}}
\newcommand{\Spin}{\mathrm{Spin}}

\newcommand{\GL}{\mathrm{GL}}

\newcommand{\ig}{\scriptscriptstyle}

\title{Geometric transitions with Spin(7) holonomy via a dynamical system}

\author{Fabian Lehmann \vspace{-3mm}}
\affil{Simons Center for Geometry and Physics, Stony Brook University
\\ \vspace*{1mm}
\textup{flehmann@scgp.stonybrook.edu}}
\date{\today}

\begin{document}

\maketitle

\begin{abstract}
We clarify the global geometry of two 1-parameter families of cohomogeneity one Spin(7) holonomy metrics with generic orbit the Aloff--Wallach space $N(1,-1) \cong \SU(3)/\U(1)$ and singular orbits $S^5$ and $\CP{2}$, which at short distance were shown to exist by Reidegeld.
The two families fit into the geography of previously known families of
cohomogeneity one metrics with exceptional holonomy and provide a Spin(7) analogue of the well-known conifold transition in the setting of Calabi--Yau 3-folds.
Furthermore, we discover that there is another transition to families of Spin(7) holonomy metrics which have a similar asymptotic behaviour on one end, but are singular on the other end.
We obtain our results by relating the Spin(7)-equations to a simple dynamical system on a 3-dimensional cube.
\end{abstract}

\tableofcontents

\section{Introduction}

Spin(7) and $\mathrm{G}_2$ are the two exceptional holonomy groups in Berger's classification of holonomy groups of Riemannian manifolds. Their significance lies in the fact that exceptional holonomy metrics are Ricci-flat, but not K\"ahler.
The first example of a metric with holonomy Spin(7) was given by Bryant in 1987 \cite{Bryant}. Two years later Bryant--Salamon found the first complete example \cite{BS}. In the 1990s most work in the area concentrated on compact manifolds \cite{big-joyce}. 
The search for metrics with exceptional holonomy has a different flavour in the non-compact setting compared to the compact setting.
A Bochner type argument shows that every Killing field on a compact Ricci-flat manifold is parallel. Therefore, 
compact irreducible manifolds with exceptional holonomy do not admit any continuous symmetries.  All known constructions in the compact setting rely on perturbative techniques starting from some degenerate limit. In contrast, symmetry reduction methods are a powerful approach in the non-compact setting. As all homogeneous Ricci-flat metrics are flat, the strongest reduction are symmetries with cohomogeneity one, i.e. where generic orbits have codimension one. The condition that a Spin(7)-structure is torsion-free reduces from a non-linear PDE system to a non-linear ODE system. The examples of an incomplete and complete Spin(7) holonomy metric by Bryant and Bryant--Salamon, respectively,  both have a cohomogeneity one symmetry.

An important geometric aspect of complete non-compact Spin(7)-manifolds is the asymptotic behaviour. The complete example given by Bryant--Salamon, which lives  
on the bundle $\mathbf{S}_+(S^4)$ of positive spinors on the 4-sphere,  is an \textit{asymptotically conical} (AC) manifold: at infinity the geometry converges to a cone with holonomy Spin(7). In this sense it is a Spin(7)-analogue of the famous Eguchi--Hanson hyperk\"ahler metric on $T^*S^2$.
The condition that the holonomy of a Riemannian cone is contained in Spin(7) is equivalent to the condition that the metric on the link is induced by a nearly parallel $\mathrm{G}_2$-structure. The link of the asymptotic cone of the Bryant--Salamon example is the ``squashed'' 7-sphere.

In the early 2000s the study of cohomogeneity one Spin(7)-manifolds gained fresh impetus by the work of Cveti\v{c}--Gibbons--L\"u--Pope. In \cite{CGLP1} they look for further examples with generic orbit $S^7$.
In their study a new type of asymptotics emerges: \textit{asymptotically locally conical} (ALC) manifolds at infinity locally look like the product of a cone and a circle of fixed size.
These are analogues of \textit{asymptotically locally flat} (ALF) hyperk\"ahler 4-manifolds such as the well-known Taub-NUT metric on $\mathbb{C}^2$.
On each of $\mathbb{R}^8$ and $\mathbf{S}_+(S^4)$ they find an explicit ALC Spin(7) holonomy metric, which they call the $\mathbb{A}_8$ and $\mathbb{B}_8$ metric, respectively.
In \cite{CGLP2} they consider a more general ansatz.
Based on numerics, they suggest that the $\mathbb{B}_8$ metric 
is part of a family of Spin(7) holonomy metrics  in a neighbourhood of $S^4\subset \mathbf{S}_+(S^4)$, which depends up to scale on one parameter $q$ and exhibits a behaviour as sketched in  diagram \ref{figure}.

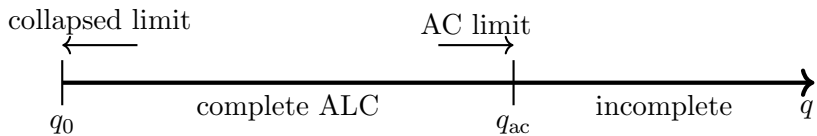
\begin{figure}[h]
\centering
\begin{tikzpicture}
\draw[thick] (0,.3) -- (0, -.3) node[below]{$q_0$};
\draw[ultra thick, ->] (0,0) --node[below]{complete ALC} (6,0)-- node[below]{incomplete} (10,0);
\draw[] (9.9, -.1) node[below]{$q$};
\draw[thick] (6,.3) -- (6,-.3) node[below]{$q_{\mathrm{ac}}$};
\draw[<-,thick] (0, .5) -- node[above]{collapsed limit}(1, .5);
\draw[->,thick] (5,.5)--node[above]{AC limit}(6, .5);
\end{tikzpicture}
\caption{Typical behaviour of a 1-parameter family of cohomogeneity one metrics with exceptional holonomy defined in a neighbourhood of the singular orbit.}
\label{figure}
\end{figure}

\noindent
In the interior of an interval $(q_0,q_{\mathrm{ac}})$ the torsion-free Spin(7)-structures are complete and ALC. At the explicit value $q_{\mathrm{ac}}$ the asymptotic geometry transitions from ALC to AC, the so-called \textit{AC limit}. This AC Spin(7)-manifold is the classical example by Bryant--Salamon. At the other endpoint of the interval the asymptotic circle of the ALC manifolds shrinks  and disappears in the limit. The family converges in the Gromov--Hausdorff topology to the Bryant--Salamon AC $\mathrm{G}_2$ holonomy metric on $\Lambda^2_{-}(S^4)$. This is called the \textit{collapsed limit}. They predict a similar 1-parameter family on $K_{\CP{3}}$, which they call the $\mathbb{C}_8$ family. Here the AC manifold, which appears at the AC limit, is asymptotic to the cone over a finite quotient of $S^7$ equipped with the round metric. It has first been discovered by Calabi \cite{Calabi} and its holonomy is $\SU(4)$ rather than Spin(7). The existence of the $\mathbb{B}_8$ and $\mathbb{C}_8$ families has been established by Bazaikin (cf. \cite{bazaikin1, bazaikin2}).

Cohomogeneity one cones with holonomy Spin(7), or equivalently homogeneous nearly parallel $\mathrm{G}_2$-manifolds with one Killing spinor, have been classified in \cite{NP-G2}. Apart from the ``squashed'' $S^7$ and the isotropy irreducible space $\SO(5)/\SO(3)$, all other examples live on the Aloff--Wallach spaces. 
If $(k,l)$ is a pair of integers which are not both zero, $\U(1)$ can be embedded via $e^{i\theta}\mapsto \mathrm{diag}(e^{i k\theta},e^{i l\theta},e^{-i(k+l)\theta})$ into the maximal torus of diagonal matrices in $\SU(3)$. 
Denote this subgroup by $\U(1)_{k,l}$. The Aloff--Wallach space $N(k,l)$ is the quotient $\SU(3)/\U(1)_{k,l}$. Each Aloff--Wallach space carries a homogeneous nearly parallel $\mathrm{G}_2$-structure.

Bazaikin \cite{bazaikin2} considered cohomogeneity one Spin(7)-structures with generic orbit $N(1,1)$ or one of the related Aloff--Wallach spaces $N(1,-2)$ and $N(-2,1)$. He again finds 1-parameter families which behave as sketched in Figure \ref{figure}. If the generic orbit is $N(1,1)$, the resulting space is the orbifold $T^*\CP{2}/\mathbb{Z}_2$.
The AC limit is obtained  by replacing $S^4$ with the non-spin manifold $\CP{2}$ in Bryant--Salamon's construction of an AC Spin(7) holonomy metric on the bundle of positive spinors on $S^4$. 
If the generic orbit is $N(1,-2)$, the underlying space is the canonical bundle of the flag manifold $F_3=\SU(3)/\U(1)^2$. Similarly to the $\mathbb{C}_8$ family, the AC limit has been constructed earlier by Calabi and has holonomy $\SU(4)$.

In this article we consider cohomogeneity one Spin(7) holonomy metrics with generic orbit isomorphic to $N(1,-1)$. Our motivation stems from the collapsed limit sketched in Figure \ref{figure}.
In \cite{BS} Bryant--Salamon constructed an AC $\mathrm{G}_2$ holonomy metric on $\Lambda^2_{-}\CP{2}$ which is asymptotic to the cone over the homogeneous nearly K\"ahler structure on the flag manifold $F_3=\SU(3)/\U(1)^2$. Because $N(1,-1)$ is a circle bundle over  $F_3$ and $(0,\infty)\times N(1,-1)$ can topologically be completed by adding a $S^5$, which is a circle bundle over $\CP{2}$, it is plausible that there exist torsion-free ALC Spin(7)-structures which collapse to the Bryant--Salamon metric on $\Lambda^2_{-}\CP{2}$. Very close to the collapsed limit, this has recently been proved independently by Foscolo \cite{Foscolo} with PDE methods.

We now turn to the formulation of our two main results.

\begin{thmx}
\label{theorem1}
Denote the adjoint bundle of the principal $\SU(2)$-bundle
\\ 
$\SU(3)\rightarrow \SU(3)/\SU(2)\cong S^5$ by $M_{S^5}$. 
$\SU(3)\times \U(1)$
acts on $M_{S^5}$ with cohomogeneity one, and  the generic orbit is the Aloff--Wallach space $N(1,-1)$.

There exists a 1-parameter family (up to scale) $\Psi_{\mu}$, $\mu\in(0,\infty)$, of $\SU(3)\times \U(1)$-invariant Spin(7) holonomy metrics in a neighbourhood of $S^5$ in $M_{S^5}$ and a distinguished parameter $\mu_{\mathrm{ac}} > 0$ such that 
\begin{itemize}
\item $\Psi_{\mu}$ is complete on $M_{S^5}$ and asymptotically locally conical (ALC) if $\mu\in(0,\mu_{\mathrm{ac}})$,
\item $\Psi_{\mu}$ is complete on $M_{S^5}$ and asymptotically conical (AC) if $\mu = \mu_{\mathrm{ac}}$,
\item $\Psi_{\mu}$ is incomplete if $\mu \in (\mu_{\mathrm{ac}},\infty)$.
\end{itemize}
$\Psi_{\mu_{\mathrm{ac}}}$ is asymptotic to the Spin(7)-cone over the unique $\SU(3)\times \U(1)$-invariant nearly parallel $\mathrm{G}_2$-structure on $N(1,-1)$.
\end{thmx}

\newpage

\begin{thmx}
\label{thmB}
Denote the universal quotient bundle of $\CP{2}$ by $M_{\CP{2}}$: the fibre of  $M_{\CP{2}}$ at $l\in \CP{2}$, which corresponds to a 1-dimensional linear subspace of $\mathbb{C}^3$, is the quotient $\mathbb{C}^3/l$.
$\SU(3)\times \U(1)$ acts on $M_{\CP{2}}$ with cohomogeneity one, and  the generic orbit is the Aloff--Wallach space $N(1,0)$, which is $\SU(3)$-equivariantly diffeomorphic to $N(1,-1)$.

There exists a 1-parameter family (up to scale) $\Upsilon_{\tau}$, $\tau\in \mathbb{R}$, of $\SU(3)\times \U(1)$-invariant Spin(7) holonomy metrics in a neighbourhood of $\CP{2}$ in $M_{\CP{2}}$ and a distinguished parameter $\tau_{\mathrm{ac}}\in\mathbb{R}$ such that
\begin{itemize}
\item 
$\Upsilon_{\tau}$ is complete on $M_{\CP{2}}$ and asymptotically locally conical (ALC) if $\tau\in(-\infty,\tau_{\mathrm{ac}})$,
\item $\Upsilon_{\tau}$ is complete on $M_{\CP{2}}$ and asymptotically conical (AC) if $\tau = \tau_{\mathrm{ac}}$,
\item $\Upsilon_{\tau}$ is incomplete if $\tau \in (\tau_{\mathrm{ac}},\infty)$.
\end{itemize}
$\Upsilon_{\tau_{\mathrm{ac}}}$ is asymptotic to the Spin(7)-cone over the unique $\SU(3)\times \U(1)$-invariant nearly parallel $\mathrm{G}_2$-structure on $N(1,-1)$.
The zero section $\CP{2}\subset M_{\CP{2}}$ is a Cayley submanifold with respect to $\Upsilon_{\tau}$ for all $\tau$. The cohomology class of $\Upsilon_{\tau}$ is non-trivial and does not depend on $\tau$.
\end{thmx}

Again the 1-parameter families $\Psi_{\mu}$ and $\Upsilon_{\tau}$
exhibit a behaviour as illustrated in Figure \ref{figure}.
$\Psi_{\mathrm{ac}}$ and $\Upsilon_{\mathrm{ac}}$ are the first complete AC Spin(7) holonomy metrics proven to exist since Bryant--Salamon's original example on $\mathbf{S}_{+}S^4$. Theorems \ref{theorem1} and \ref{thmB} were previously conjectured by Cveti\v{c}--Gibbons--L\"u--Pope \cite{CGLP2} and Gukov--Sparks--Tong \cite{Gukov-Sparks-Tong}. 
Similar to the $\mathbb{B}_8$-family, 
there exists a specific parameter $\tau^* < \tau_{\mathrm{ac}}$
for which the ALC Spin(7) holonomy metric $\Upsilon_{\tau^*}$ has an explicit expression, which was found by
Cveti\v{c}--Gibbons--L\"u--Pope \cite{CGLP2}, 
Gukov--Sparks \cite{Gukov-Sparks} and Kanno--Yasui \cite{Kanno-Yasui}. To the knowledge of the author, there is no explicit expression known of $\Psi_{\mu}$ for any $\mu$.

The 1-parameter families $\Psi_{\mu}$ and $\Upsilon_{\tau}$ have first been constructed by Reidegeld \cite{Reidegeld-paper} in a neighbourhood of the exceptional orbits $S^5$ and $\CP{2}$, respectively. Whether they give rise to complete metrics was left open.
A non-compact cohomogeneity one space with principal orbit isomorphic to $N(1,-1)$ can also be topologically completed by adding as a singular orbit either the flag manifold $F_3$ or $\SU(3)/\SO(3)$, the space of linear special Lagrangian subspaces of $\mathbb{C}^3$ \cite[Lemma 5.4.2]{reidegeld}. However, Reidegeld showed that there can be no torsion-free Spin(7)-structures on the resulting spaces which are invariant by all of $\SU(3)\times \U(1)$ \cite[pp. 177-178, pp. 192-197]{reidegeld}.

There is a qualitative difference in the collapsed limit of the two families $\Psi_{\mu}$ and $\Upsilon_{\tau}$. Because $S^5$ is a circle bundle over $\CP{2}$, $M_{S^5}$ globally has the structure of a circle bundle over $\Lambda^2_{-}\CP{2}$. As $\mu\rightarrow 0$, the collapse occurs with bounded curvature similar to the well-known collapse of Berger's sphere. In contrast, the manifold $M_{\CP{2}}$ has the structure of a circle bundle only outside the zero section. The fixed locus of the circle action is precisely the zero section. Therefore, the curvature blows up on $\CP{2}$ as $\tau\rightarrow -\infty$. Because Foscolo's analytic method \cite{Foscolo} only applies for collapse with bounded curvature, the result about the existence of a continuous family of complete ALC Spin(7) holonomy metrics on $M_{\CP{2}}$ is new.

The AC Spin(7) holonomy manifolds $(M_{S^5}, \Psi_{\mu_{\mathrm{ac}}})$
and $(M_{\CP{2}}, \Upsilon_{\tau_{\mathrm{ac}}})$ resemble the well-known conifold transition of Calabi--Yau 3-folds: the \textit{smoothing} $T^*S^3$ of the conifold
$\{(z_1,z_2,z_3,z_4)\in\mathbb{C}^4|\ z_1^2+z_2^2+z_3^2+z_4^2 = 0 \}$  and its \textit{small resolution} $\mathcal{O}(-1)\oplus \mathcal{O}(-1) \rightarrow \CP{1}$ are topologically different spaces which carry  cohomogeneity one AC Calabi--Yau metrics asymptotic to the same Calabi--Yau cone, the conifold.
In analogy to the Calabi--Yau setting, the Spin(7) conifold transition has been conjectured by
Gukov--Sparks--Tong \cite{Gukov-Sparks-Tong}. 
The collapsed limit 
also has a lower dimensional analogue. The $\mathbb{D}_7$ family of cohomogeneity one $\mathrm{G}_2$ holonomy metrics on $\mathbf{S}(S^3) = S^3\times \mathbb{R}^4$ collapses with bounded curvature to the AC Calabi--Yau metric on the small resolution of the conifold, while the curvature blows up on the zero section $S^3$ as the $\mathbb{B}_7$ family of cohomogeneity one $\mathrm{G}_2$ holonomy metrics on $\mathbf{S}(S^3) = S^3\times \mathbb{R}^4$ collapses to the AC Calabi--Yau metric on the smoothing of the conifold. 
Again Gukov--Sparks--Tong conjecture that the families $\Psi_{\mu}$ and $\Upsilon_{\tau}$ provide an analogue of this phenomenon.

\subsection*{Strategy to prove Theorems \ref{theorem1} and \ref{thmB} and relation to a dynamical system}
The key step in proving Theorems \ref{theorem1} and \ref{thmB} is to establish the existence of the AC spaces. The behaviour of the remaining family members, which lead to ALC and incomplete metrics, can be deduced by a comparison argument. 
In Bazaikin's work on the $\mathbb{B}_8$ family, the $\mathbb{C}_8$ family, and on cohomogeneity one Spin(7)-manifolds with generic orbit isomorphic to $N(1,1)$, the AC limit was known beforehand. Moreover, in these examples the AC spaces enjoy additional symmetry as compared to other family members and are given by an explicit expression. Foscolo--Haskins--Nordstr\"om \cite{FHN2} consider problems in the context of cohomogeneity one $\mathrm{G}_2$-manifolds in which the AC limits are not known beforehand. They solve this problem by ``shooting from infinity''.
By deforming the $\mathrm{G}_2$-cone, they construct a family of AC ends,
i.e. torsion-free AC $\mathrm{G}_2$-structures defined outside a compact subset. Building upon a good understanding of the space of AC ends and their asymtpotics to the cone, they show that one of these AC ends 
extends smoothly over a singular orbit and thus gives a complete AC $\mathrm{G}_2$ holonomy metric.
Homogeneous nearly K\"ahler manifolds, which are the links of cohomogeneity one $\mathrm{G}_2$-cones, are normal and standard homogeneous spaces, as are the squashed 7-sphere and the 3-Sasakian metrics on $S^7$ and $N(1,1)$, which are the links of the cones in Bazaikin's work. In this sense, we are in a less symmetric setting, as the homogeneous nearly parallel $\mathrm{G}_2$-structure on $N(1,-1)$ is neither standard nor normal. We want to avoid computations involving them and therefore pursue a different strategy as in \cite{FHN2}.
Our method allows us to deduce the existence of the AC metric
from the existence of ALC solutions and incomplete metrics.
More specifically, we show that the corresponding sets of parameters which give rise to complete ALC solutions and incomplete metrics each are  open and non-empty. In addition, we manage to show that the complement of these two sets are precisely the parameters which give rise to complete AC solutions. In particular, because the set of all parameters is connected, we deduce the existence of an AC solution.

In the following we give a brief overview of how we carry out this strategy.
Outside the singular orbit a torsion-free $G$-invariant Spin(7)-structure can be interpreted as a trajectory in the space $\mathcal{S}$ of co-closed $G$-invariant $\mathrm{G}_2$-structures on the principal orbit $G/H$ given as a solution of the evolution equation \eqref{fundamental evo eqn}. 
Our main emphasis is the choice of ``good coordinates'' on the state space $\mathcal{S}$, which in our case  is 4-dimensional. 
The geometric setting suggests to use particular coordinates $(a,b,c,f)$
which 
allow us to conveniently read off the asymptotic behaviour of  complete
ALC and AC solutions.
The resulting ODE system with respect to these coordinates has been derived earlier by \cite{CGLP2} and \cite{Reidegeld-paper}.
However, from an analytic point of view it is difficult to work with the $(a,b,c,f)$-coordinates directly as the ODE system is highly non-linear and coupled. But the ODE system has the favourable property that its right-hand side is a homogeneous expression in $(a,b,c,f)$. This allows us to consider the ODE system in projective space, thereby eliminating one dimension. 
Through this projectivization, we thus arrive at a ``dynamic'' description of our Spin(7)-equations, where singular orbits and asymptotic models are given by fixed points of the dynamical system and complete torsion-free Spin(7)-structures correspond to trajectories connecting these fixed points. In particular, we obtain a complete list of possible asymptotic geometries. This approach is inspired by Atiyah--Hitchin's work \cite{AH} on gravitational instantons. 
In our setting, however, a further coordinate change on projective space is needed to gain control over all three remaining functions. 
In this new set of coordinates $(X,Y,Z)$ the trajectories of the families $\Psi_{\mu}$ and $\Upsilon_{\tau}$ 
emanate from fixed points located on vertices of the cube $\mathcal{W}=(0,1)\times(0,1)\times (0,5/4)$. Initially they enter the interior of the cube
which they can exit only at the face $Y=0$, where the associated Spin(7)-structure degenerates.
The ODE system in $(X,Y,Z)$-space has the key properties that there are no intrinsic singularities and complete flow lines contained in compact regions need to converge to a fixed point. By analysing all fixed points in the closure of $\mathcal{W}$, we see that the trajectories associated with the families $\Psi_{\mu}$ and $\Upsilon_{\tau}$ can only converge to a fixed point representing the asymptotic model of an ALC space or to the fixed point which corresponds to the Spin(7)-cone over $N(1,-1)$. Therefore, the Spin(7)-metrics $\Psi_{\mu}$ and $\Upsilon_{\tau}$ are either ALC, AC, or incomplete. Because the ALC end is a sink and 
exiting the cube at the face $Y=0$ is an open condition, the sets of parameters giving rise to ALC and incomplete solutions must both be open. To see that they are non-empty we study their limiting flow lines, e.g. as the parameter $\mu$ approaches $0$ or $\infty$. 

The space of flow lines which converge to the fixed point corresponding to the Spin(7)-cone is homeomorphic to a circle.
The two complete AC spaces $(M_{S^5}, \Psi_{\mu_{\mathrm{ac}}})$
and $(M_{\CP{2}}, \Upsilon_{\tau_{\mathrm{ac}}})$, which we have thus proven to exist, each are represented by a point on this circle. 
There are two flow lines which emanate from this fixed point.
With a more quantitative approach, which relies on studying extrema of the function $Y$, we can prove that one of these two flow lines  converges towards the fixed point corresponding to the ALC end. The geometric interpretation is a \textit{conically singular} (CS) ALC space. Indeed, we prove the existance of an $\SU(3)\times \U(1)$-invariant CS ALC Spin(7)-metric with principal orbit $N(1,-1)$.
Previously, other examples of CS Spin(7) metrics were given by the $\mathbb{A}_8$ metric on $\mathbb{R}^8$ and variations of its construction by replacing the $S^7$ with an arbitrary 3-Sasakian 7-manifold.

\begin{thmx}
\label{Theorem 2}
There exists a 1-parameter family  $\Psi^{\mathrm{cs}}_{\lambda}$, $\lambda\in\mathbb{R}$, of $\SU(3)\times \U(1)$-invariant Spin(7) holonomy metrics on $(0,\varepsilon(\lambda))_t \times N(1,-1)$, with $\varepsilon(\lambda) > 0$ for every $\lambda\in\mathbb{R}$,
which are conically singular (CS) as $t\rightarrow 0$ asymptotic to the Spin(7)-cone over the unique $\SU(3)\times \U(1)$-invariant nearly parallel $\mathrm{G}_2$-structure on $N(1,-1)$. 
The 1-parameter family has the following properties:
\begin{itemize}
\item If $\lambda < 0$, then $\Psi_{\lambda}^{\mathrm{cs}}$ extends to $(0,\infty)_t \times N(1,-1)$, is forward complete and asymptotically locally conical (ALC) as $t\rightarrow \infty$.
\item If $\lambda = 0$, then  $\Psi_{\lambda}^{\mathrm{cs}}$ 
is the Spin(7)-cone over the unique $\SU(3)\times \U(1)$-invariant nearly parallel $\mathrm{G}_2$-structure on $N(1,-1)$.
\item If $\lambda > 0$, then $\Psi_{\lambda}^{\mathrm{cs}}$ does not extend to a forward complete metric.
\end{itemize}
For a fixed sign of $\lambda$ the Spin(7)-structures $\Psi^{\mathrm{cs}}_{\lambda}$ are related by scaling.
\end{thmx}
As the parameters $\mu$ and $\tau$ approach the AC limits $\mu_{\mathrm{ac}}$ and $\tau_{\mathrm{ac}}$, respectively, 
the corresponding flow lines representing ALC spaces converge to 
the union two flow lines, one the respective AC space, and the other the conically singular ALC space. This has a geometric interpretation:
In the 1-parameter families $\Psi_{\mu}$ and $\Upsilon_{\tau}$
we have chosen a scale by fixing the size of the singular orbit. In this normalisation, as the parameter increases, the asymptotic circle length of the complete ALC solutions increases until it approaches infinity at the AC limit. If, however, we rescale to keep the asymptotic circle length fixed, the size of the singular orbit shrinks as the parameter increases and approaches zero at the limit. Therefore, this limit is a conically singular ALC space with the AC space bubbling off.

There is a further fixed point on one of the vertices of the cube $\mathcal{W}$, which we have not talked about so far. Unlike the fixed points corresponding to the singular orbits $S^5$ and $\CP{2}$, which both have 2-dimensional unstable manifolds, this fixed point is a source. We prove that among the 2-dimensional space of flow lines emanating from this fixed point we can find 1-parameter families which behave as $\Psi_{\mu}$ and $\Upsilon_{\tau}$. 

\begin{thmx}
\label{thm-sing}
There exists a positive constant $z_0$ such that for each $z\in(0,z_0)$ there exists a 1-parameter family $\Omega^{z}_{\kappa}, \kappa\in(0,1)$, of $\SU(3)\times\U(1)$-invariant Spin(7) holonomy metrics on $(0,\epsilon(z,\kappa))_t\times N(1,-1)$, with $\epsilon(z,\kappa) > 0$ for all $\kappa\in(0,1)$,
with the following properties: There exists $\kappa_{\mathrm{ac}}(z)\in(0,1)$ such that
\begin{itemize}
\item If $\kappa< \kappa_{\mathrm{ac}}(z)$, $\Omega^z_{\kappa}$ extends to $(0,\infty)\times N(1,-1)$, is forward complete and asymptotically locally conical (ALC) as $t\rightarrow\infty$.
\item
If $\kappa = \kappa_{\mathrm{ac}}(z)$,
$\Omega^z_{\kappa}$ extends to $(0,\infty)\times N(1,-1)$, is forward complete and asymptotically conical  (AC) as $t\rightarrow\infty$.
\item 
If $\kappa > \kappa_{\mathrm{ac}}(z)$, then $\Omega^z_{\kappa}$ is forward incomplete.
\end{itemize}
All Spin(7) holonomy metrics $\Omega^z_{\kappa}$ have a singularity as $t\rightarrow 0$.
\end{thmx}

Again each of the families $\Omega^z_{\kappa}$ has a collapsed limit. 
As $\kappa\rightarrow 0$ the trajectories corresponding to $\Omega^z_{\kappa}$ converge to the diagonal flow line in the face $Z=0$. 
This explicit solution is a singular version of the Bryant--Salamon $\mathrm{G}_2$ holonomy metric on $\Lambda^2_{-}\CP{2}$. As $t\rightarrow 0$ one of the metric coefficient functions blows-up and the others converge to 0.

\begin{figure}
\centering
\includegraphics[scale=0.6]{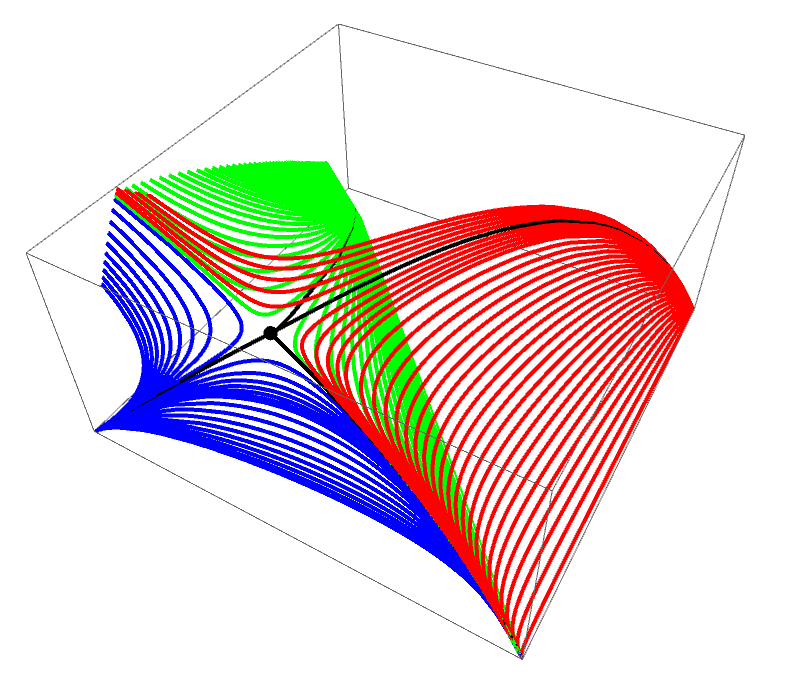}
\caption{This 3-dimensional, numerical plot summarises our work and illustrates the various geometric transitions and limiting phenomena. It depicts the families $\Psi_{\mu}$ (red), $\Upsilon_{\tau}$ (blue) and $\Omega^{0.15}_{\kappa}$ (green).
The flow lines of the three AC spaces and the CS ALC space are coloured black.}
\label{dynamical-system-pic}
\end{figure}

\paragraph*{Acknowledgements} This work is a result of the author's PhD thesis and 
was supported by the Engineering and Physical Sciences Research Council [EP/L015234/1], 
the EPSRC Centre for Doctoral Training in Geometry and Number Theory (The London School of Geometry and Number Theory), University College London.
I want to thank my PhD supervisors Mark Haskins, Jason Lotay and Lorenzo Foscolo for their support, and my PhD examiners Johannes Nordstr\"om and Simon Salamon for their helpful comments. In particular, I want to thank Simon Salamon for generously sharing with me the Mathematica code which I have used to create figures  \ref{dynamical-system-pic} and \ref{2Dplot}.

\section{Preliminaries}

\label{section-pre}

\subsection{Special geometric structures in dimensions 6, 7 and 8}

In this section we provide the essential background on manifolds equipped with special geometric structures. For more details we refer to \cite{small-joyce} and \cite{Salamon-book}.
Even though the main focus of this paper is on Spin(7)-structures, we will also need the structure groups $\SU(3)$ and $\mathrm{G}_2$.

\begin{definition}
An $\SU(3)$-structure on a 6-dimensional manifold $M$ is a pair $(\omega,\Omega)$ of a non-degenerate 2-form $\omega$ and a complex volume form $\Omega$ satisfying the algebraic constraints
\begin{align*}
\omega\wedge \Omega = 0, \quad
\frac{1}{6}\omega^3 = \frac{1}{4} \mathrm{Re}\ \Omega\wedge \mathrm{Im}\ \Omega.
\end{align*}
Equivalently, at each point $p\in M$ there exists a linear isomorphism between $T_p M$ and $\mathbb{C}^3$ which identifies 
$\omega|_p$ and $\Omega|_p$ with
\begin{align*}
\omega_0
=
\frac{i}{2}
(dz_1\wedge d\bar{z}_1+dz_2\wedge d\bar{z}_2 + dz_3\wedge d\bar{z}_3), 
\quad 
\Omega_0
=
dz_1\wedge dz_2 \wedge dz_3.
\end{align*} 
The pair $(\omega,\Omega)$ reduces the structure group of the frame bundle of $M$ to $\SU(3)$ by considering the subbundle 
$\{u\colon \mathbb{C}^3 \xrightarrow{\sim} T_pM\ \text{linear}\ |\
u^*(\omega|_p,\Omega|_p)=(\omega_0,\Omega_0) \}$.
\end{definition}

\begin{definition}
$\mathrm{G}_2$ is the subgroup of $\GL(7,\mathbb{R})$ which preserves the 3-form
\begin{align*}
\varphi_0
=
dx_{123}+dx_{145}+dx_{167}+dx_{246}-dx_{257}-dx_{347}-dx_{356}.
\end{align*}
Here $(x_1,\dots,x_7)$ are coordinates on $\mathbb{R}^7$
and  we denote $dx_i\wedge dx_j\cdots \wedge dx_l$ by $dx_{ij\dots l}$. Now let $M$ be an oriented 7-manifold. We say that a 3-form $\varphi$ is \textit{positive} if at each point $p\in M$ there exists an orientation-preserving linear isomorphism between $T_pM$ and $\mathbb{R}^7$ such that $\varphi|_p$ is identified with $\varphi_0$. We refer to $\varphi$ as a \textit{$\text{G}_2$-structure}. $\varphi$ reduces the structure group of the frame bundle of $M$ to $\mathrm{G}_2$ by considering the subbundle 
$\{u\colon \mathbb{R}^7 \xrightarrow{\sim} T_pM\ \text{linear}\ |\
u^*\varphi_p = \varphi_0 \}$. $\mathrm{G}_2$ is a subgroup of $\SO(7)$. In particular, $\varphi$ induces in a purely algebraic way a Riemannian metric $g$. 
The Hodge star operator induced by $g$ gives the Hodge dual 4-form $*\varphi$. The condition that the holonomy group of $g$ is contained in $\mathrm{G}_2$ is equivalent to $d\varphi = 0$ and $d*\varphi=0$. We call such a $\mathrm{G}_2$-structure \textit{torsion-free}. 
This is a non-linear condition as the Hodge star operator depends in a non-linear way on $\varphi$.
\end{definition}

\begin{definition}
Similar to $\mathrm{G}_2$ the spin group $\Spin(7)$ can be characterised as the stabiliser of the 4-form 
\begin{align*}
\psi_0
=&
dx_{1234}+dx_{1256}+dx_{1278}+dx_{1357}-dx_{1368}
-dx_{1458}-dx_{1467}
\\
&-dx_{2358}-dx_{2367}-dx_{2457}
+dx_{2468}+dx_{3456}+dx_{3478}+dx_{5678}
\end{align*}
where $(x_1,\dots,x_8)$ are coordinates on $\mathbb{R}^8$.
Now let $M$ be an oriented 8-manifold. We say that a 4-form $\psi$ is \textit{admissible} if at each point $p\in M$ there is an orientation-preserving isomorphism between $T_pM$ and $\mathbb{R}^8$ which identifies $\psi|_p$ with $\psi_0$. 
We refer to $\psi$ as a \textit{Spin(7)-structure}. $\psi$ reduces the structure group of the frame bundle of $M$ to Spin(7) by considering the subbundle 
$\{u\colon \mathbb{R}^8 \xrightarrow{\sim} T_p M\ \text{linear}\ |\
u^*\psi_p = \psi_0 \}$. Spin(7) is a subgroup of $\SO(8)$. In particular, $\psi$ induces in a purely algebraic way a Riemannian metric $g$. 
With respect to this metric the 4-form $\psi$ is self-dual. The condition that the holonomy group of the induced metric is contained in $\Spin(7)$ is equivalent to $d\psi=0$.
We call such a $\Spin(7)$-structure \textit{torsion-free}.
\end{definition}

While the orbit of $\varphi_0$ under the action of $\GL(7,\mathbb{R})$ is an open subset of $\Lambda^3(\mathbb{R}^7)^*$, the orbit of $\psi_0$ under the action of $\GL(8,\mathbb{R})$ is a nonlinear subspace of $\Lambda^4(\mathbb{R}^8)^*$ of codimension 27. It follows that the space of admissible forms is non-linear and the condition $d\psi = 0$ is a non-linear system of PDEs. 

\begin{remark}
For our purposes an important fact is that there is a chain of inclusions $\SU(3) \subset \mathrm{G}_2 \subset \mathrm{Spin}(7)$.
This can be seen as follows. If $(M,\Omega,\omega,h)$ is a 6-dimensional manifold equipped with an $\SU(3)$-structure, then we obtain a $\mathrm{G}_2$-structure on $M\times \mathbb{R}$ by
\begin{align}
\varphi = dt\wedge\omega + \mathrm{Re}\ \Omega,
\quad
*\varphi = \frac{1}{2}\omega^2-dt\wedge\mathrm{Im}\ \Omega,
\quad
g = dt^2 + h.
\label{G2-reduction-SU(3)}
\end{align} 
If $(M,\varphi,*\varphi,h)$ is a 7-dimensional manifold equipped with a $\mathrm{G}_2$-structure, then we obtain a Spin(7)-structure
on $M\times \mathbb{R}$ by
\begin{align}
\psi = dt\wedge \varphi + *\varphi,
\quad
g = dt^2 + h.
\label{Spin(7)-reduction-G2}
\end{align} 
\end{remark}

\subsection{Cohomogeneity one Spin(7)-manifolds}
\label{intro-coh1}

References for our brief introduction to cohomogeinity one manifolds are  \cite{Mostert,reidegeld,Reidegeld-paper}.
Let $G$ be a compact Lie group acting continuously on the connected manifold $M$. We say this action is of \textit{cohomogeneity one} 
if there exists an orbit with codimension 1.
In this case the quotient $M/G$ has to be diffeomorphic to either $S^1$,  $[0,1]$, $\mathbb{R}$ or $[0,\infty)$.
In the first two cases $M$ is compact. However, by a Bochner-type argument compact irreducible Ricci-flat manifolds cannot have any continuous symmetries. In the third case $M$ has two ends. However, by the Cheeger--Gromoll splitting theorem complete irreducible  Ricci-flat manifolds can have only one end. Therefore, in the context of complete cohomogeneity one manifolds with holonomy Spin(7) only the last case  is interesting and from now on
we only consider $M/G=[0,\infty)$. Denote by $q: M\rightarrow M/G$ the quotient map. Isotropy groups of orbits which $q$ does not map to the end point of the half-open interval $[0,\infty)$ are conjugate to one another and there exists $H\subset G$ such that $q^{-1}(0,\infty)$ is $G$-equivariantly diffeomorphic to $(0,\infty)\times G/H$. These orbits are called \textit{principal orbits}. The orbit $q^{-1}(0)$ is called the \textit{singular orbit}. Denote its isotropy group by $K$, i.e. $q^{-1}(0)=G/K$.
This allows us to write
\begin{align}
\label{decomposition-M}
M = (G/K) \cup (0,\infty)\times (G/H).
\end{align}
We can say more about the structure of $M$.
Note that $G \rightarrow G/K$ is a principal $K$-bundle. 
We can choose $H\subset K$ such that $K/H$ is diffeomorphic to a sphere. 
In fact there exists a representation $V$ of $K$ such that 
$M$ has the structure 
of the total space of the associated vector bundle $G \times_K V \rightarrow G/K$ over the singular orbit, the principal orbits $\{t\}\times G/H$ are sphere bundles over $G/K$ which foliate the vector bundle outside the zero section and the spherical fibres of the fibrations $\{t\} \times G/H \rightarrow G/K$ are isomorphic to $K/H$. 
 
We say that a Spin(7)-manifold $(M,\psi)$ is a \textit{cohomogeneity one Spin(7)-manifold} if there exists a cohomogeneity one action by some compact Lie group $G$ on $M$ such that $\psi$ is $G$-invariant. Then $G$ also preserves the induced metric.
The Spin(7)-structure $\psi$ induces on each principal orbit $\{t\}\times G/H$ a $G$-invariant $\mathrm{G}_2$-structure $(\varphi_t,h_t)$ and on $q^{-1}(0,\infty)$ the Spin(7)-structure can be recovered as 
\begin{align}
\psi &= dt\wedge \varphi_t + *\varphi_t,
\label{Spin(7)-structure from G_2 evo}
\\
g &= dt^2+h_t.
\end{align}
Here the Hodge star depends on $\varphi_t$.
The condition $d\psi=0$ for $\psi$ to be torsion-free then is equivalent to the system
\begin{subequations}
\label{tf condition}
\begin{align}
d_{\ig G/H}\! *\! \varphi_t& = 0,
\label{tf condition-static eqn}
\\
\frac{\partial}{\partial t}\!*\!\varphi_t &= d_{\ig G/H}\varphi_t.
\label{tf condition-evo eqn}
\end{align}
\end{subequations}
Here $d_{\ig G/H}$ denotes the exterior derivative on $G/H$.
The first equation is a static condition, i.e. it does not involve a derivative with respect to the parameter $t$. Therefore, we can interpret a torsion-free Spin(7)-structure on the dense subset $M-q^{-1}(0)$ as a solution of the evolution equation 
\begin{align}
\label{fundamental evo eqn}
\frac{\partial}{\partial t}\!*\!\varphi_t &= d_{\ig G/H}\varphi_t
\end{align}
in the space of co-closed, $G$-invariant $\mathrm{G}_2$-structures on the homogeneous space $G/H$. Note that that this space is finite dimensional.

How can we approach the problem of constructing a complete $G$-invariant torsion-free Spin(7)-structure on $M$? Fixing a co-closed $G$-invariant $\mathrm{G}_2$-structure $\hat{\varphi}$ on a principal orbit $\{t_0\}\times G/H$ leads to a well-defined initial value problem. By the Picard--Lindel\"of theorem there exists a torsion-free Spin(7)-structure on $(t_0-\varepsilon,t_0+\varepsilon)\times G/H$ of the form (\ref{Spin(7)-structure from G_2 evo}) with $\varphi_0=\hat{\varphi}$ for some $\varepsilon > 0$. To investigate whether this Spin(7)-structure can be extended to a complete torsion-free Spin(7)-structure, two questions have to be addressed. First, does it extend backward and close smoothly on the singular orbit? Secondly, does it extend forward over the non-compact end? In general neither question is easy to answer. In the context of non-compact cohomogeneity one Einstein metrics Eschenburg--Wang \cite{eschenburg-wang} take a different approach. They instead consider a singular initial value problem on the singular orbit.
Smooth solutions give rise to smooth Einstein metrics in a neighbourhood of the singular orbit. To investigate completeness it remains to check whether the solution extends over the non-compact end. This has become the standard approach in the construction of cohomogeneity one structures, e.g. \cite{FHN2}. 
Also see the more recent treatment by Verdiani--Ziller \cite{Verdiani-Ziller} on the problem of extending cohomogeneity one metrics smoothly over the singular orbit.
In the realm of special holonomy a simplifying assumption made by Eschenburg--Wang is often not satisfied and their approach has to be adjusted accordingly.
In particular, Reidegeld \cite{reidegeld,Reidegeld-paper} studied this singular initial value problem in the context of Spin(7)-structures.

\subsection{Asymptotic geometries}

If $(\Sigma,h)$ is a closed Riemannian manifold 
then the \textit{Riemannian cone} $(C(\Sigma),g_C)$
is the manifold $C(\Sigma)=(0,\infty)\times \Sigma$
equipped with the metric $g_C = dt^2+t^2 h$.
Here $t$ is the radial coordinate on $(0,\infty)$.
If in particular $\Sigma$ is 6-dimensional and the metric $h$ is given by 
an $\SU(3)$-structure $(\Sigma,\Omega,\omega,h)$
then we can equip $C(\Sigma)$ with a 
$\mathrm{G}_2$-structure 
given by
\begin{align*}
\varphi_C
=
t^2 dr\wedge\omega+ t^3 \mathrm{Re}\ \Omega,
\quad
*\varphi_C
=
\frac{1}{2} t^4 \omega^2 - t^3 dt\wedge \mathrm{Im}\ \Omega,  
\end{align*}
which induces the cone metric $g_C$. 
$(C(\Sigma),\varphi_C,g_C)$ is said to be a $G_2$\textit{-cone}
if the $\mathrm{G}_2$-structure is torsion-free. 
The exterior derivatives are
\begin{align*}
d\varphi_C
&=
-t^2 dt\wedge d\omega +3t^2 dt\wedge \mathrm{Re}\ \Omega
+t^3 d\, \mathrm{Re}\ \Omega,
\\
d*\varphi_C
&=
2t^3 dt\wedge\omega^2
+
t^4 d\omega\wedge\omega
+
t^3 dt\wedge d\, \mathrm{Im}\ \Omega.
\end{align*} 
Hence the condition $d\varphi=d*\varphi=0$ is equivalent to
\begin{align}
d\omega=3\mathrm{Re}\ \Omega,
\quad
d\, \mathrm{Im}\ \Omega= -2 \omega^2.
\label{NK condition}
\end{align}
This means precisely that the $\SU(3)$-structure 
on $\Sigma$ is \textit{nearly K\"ahler}. 
Because nearly K\"ahler manifolds are Einstein manifolds with 
positive scalar curvature, $\Sigma$ has to be compact by the
Bonnet--Myers theorem.

Analogously if $\Sigma$ is 7-dimensional and the metric $h$
is induced by the $\mathrm{G}_2$-structure $(\Sigma,\varphi,h)$,
then we can equip $C(\Sigma)$
with a $\Spin(7)$-structure given by
\begin{align*}
\psi_C
=
t^3 dt\wedge\varphi + t^4 *\!\varphi,
\end{align*}
which induces the cone metric $g_C$. 
$(C(\Sigma),\psi_C)$ is said to be a $\Spin(7)$-cone
if $\psi_C$ is torsion-free. 
The exterior derivative is given by
\begin{align*}
d\psi_C
=
-t^3 dt\wedge d\varphi
+
4 t^3 dt\wedge *\varphi 
+
t^4 d\!*\!\varphi.
\end{align*}
Hence the condition $d\psi_C=0$ is equivalent
to
\begin{align*}
d\varphi = 4 *\!\varphi.
\end{align*}
This means that the $\mathrm{G}_2$-structure
on $\Sigma$ is \textit{nearly parallel}.
Again nearly parallel $\mathrm{G}_2$-manifolds
are Einstein manifolds with positive scalar curvature so they have to be compact.
If the link is the 7-sphere with the round metric,
then the cone is the Euclidean $\mathbb{R}^8$ with the standard Spin(7)-structure. Apart from its quotients this is the only Spin(7)-cone with trivial holonomy. All other Spin(7)-cones need to have  holonomy group $\Sp(2)$, $\SU(4)$ or Spin(7). 
If the holonomy equals Sp(2) the link is a \textit{3-Sasakian}
manifold, if it equals SU(4) then the link must be a 7-dimensional \textit{Sasaki--Einstein} manifold
and in the last case we say the nearly parallel 
$\mathrm{G}_2$-structure is \textit{proper}.

We can now give a precise definitions of CS, AC and ALC Spin(7)-manifolds.

\begin{definition}
\label{def-Spin(7)-cone}
Let $(C(\Sigma),\psi_C, g_C)$ be a $\Spin(7)$-cone over the 
nearly parallel $\mathrm{G}_2$-manifold $(\Sigma,\varphi,h)$. A $\Spin(7)$-manifold $(M,\psi,g)$ has an isolated conical singularity 
asymptotic to the cone $(C(\Sigma),\psi_C, g_C)$ with rate $\nu\in(0,\infty)$ if there exists an open subset $U \subset M$
and a diffeomorphism 
\begin{align*}
F: (0,\varepsilon)\times \Sigma \subset C(\Sigma) \rightarrow U
\end{align*}
for some $\varepsilon > 0$ such that 
\begin{align*}
|\nabla_C^j(F^*\psi-\psi_C)|_{g_C}=\mathcal{O}(t^{\nu-j})
\quad
\text{for all}\ j\in\mathbb{N}_0\ \text{as}\ t\rightarrow 0.
\end{align*}
In particular this implies
\begin{align*}
|\nabla_C^j(F^*g-g_C)|_{g_C}=\mathcal{O}(t^{\nu-j})
\quad
\text{for all}\ j\in\mathbb{N}_0\ \text{as}\ t\rightarrow 0.
\end{align*}
We say that $(M,\psi,g)$ is a \textit{conically singular} (CS) Spin(7)-manifold.
\end{definition}

\begin{definition}
\label{def-AC}
Let $(C(\Sigma),\psi_C, g_C)$ be a $\Spin(7)$-cone over the 
nearly parallel $\mathrm{G}_2$-manifold $(\Sigma,\varphi,h)$. A $\Spin(7)$-manifold $(M,\psi,g)$ is an \textit{asymptotically conical} (AC) $\Spin(7)$-manifold asymptotic to $(C(\Sigma),\psi_C)$
with rate $\nu\in(-\infty,0)$ if there exist a compact subset
$K\subset M$ and a diffeomorphism 
\begin{align*}
F: (R,\infty)\times \Sigma \subset C(\Sigma) \rightarrow M-K
\end{align*}
for some $R>0$ such that 
\begin{align*}
|\nabla_C^j(F^*\psi-\psi_C)|_{g_C}=\mathcal{O}(t^{\nu-j})
\quad
\text{for all}\ j\in\mathbb{N}_0\ \text{as}\ t\rightarrow \infty.
\end{align*}
In particular this implies
\begin{align*}
|\nabla_C^j(F^*g-g_C)|_{g_C}=\mathcal{O}(t^{\nu-j})
\quad
\text{for all}\ j\in\mathbb{N}_0\ \text{as}\ t\rightarrow \infty.
\end{align*}
\end{definition}

\begin{definition}
Let $(C(\Sigma),\varphi_C,g_C)$ be a $\mathrm{G}_2$-cone over the nearly K\"ahler manifold $(\Sigma,\Omega,\omega,h)$,
$\ell$ a positive constant
and $p:P\rightarrow C(\Sigma)$ a $\U(1)$-principal bundle
with a connection $\theta\in\Omega^1(P)$ which gives rise
to a $\Spin(7)$-structure on $P$ via
$\psi_P=\ell \theta\wedge\varphi_C +*\varphi_C$
with associated metric $g_P=g_C+\ell^2\theta^2$. 
A $\Spin(7)$-manifold $(M,\psi,g)$ is said to be an \textit{asymptotically locally conical} (ALC)
$\Spin(7)$-manifold asymptotic to $(P,\psi_P,g_P)$
with rate $\nu \in (-\infty,0)$ and asymptotic circle length
$\ell$
if there exists a compact subset $K\subset M$ and (possibly for a double cover of $M-K$) a 
diffeomorphism
\begin{align*}
F: p^{-1}((R,\infty)\times \Sigma)\subset P \rightarrow M-K
\end{align*}
for some $R > 0$ such that 
\begin{align*}
|\nabla_P^j(F^*\psi-\psi_P)|_{g_P}
=
\mathcal{O}(t^{\nu-j})\quad
\text{for all}\ j\in\mathbb{N}_0\ \text{as}\ t\rightarrow \infty.
\end{align*}
In particular this implies
\begin{align*}
|\nabla_P^j(F^*g-g_P)|_{g_P}=\mathcal{O}(t^{\nu-j})
\quad
\text{for all}\ j\in\mathbb{N}_0\ \text{as}\ t\rightarrow \infty.
\end{align*}
\end{definition}

\section{Spin(7)-metrics with Principal Orbit $N(1,-1)$}
\label{section-principal orbit}

\subsection{The Aloff--Wallach space $N(1,-1)$}

For every pair $(k,l)$ of integers which are not both zero, 
$\U(1)$ can be embedded in the maximal torus of diagonal matrices in $\SU(3)$ as 
\begin{align}
e^{i\theta}\mapsto
\begin{pmatrix}
e^{i k \theta} & 0 & 0\\
0 & e^{i l \theta} & 0\\
0 & 0 & e^{-i(k+l)\theta}
\end{pmatrix}.
\label{def-U_k,l}
\end{align}
We also denote this subgroup of $\SU(3)$ by $\U(1)_{k,l}$.
The Aloff--Wallach space $N(k,l)$ is the homogeneous space $\SU(3)/\U(1)_{k,l}$. 
We work with the following basis of $\mathfrak{su}(3)$:
\begin{equation*}
\begin{array}{r@{\hspace{3pt}}c@{\hspace{3pt}}l@{\hspace{20pt}}r@{\hspace{3pt}}c@{\hspace{3pt}}l}
E_1 & = &
\begin{pmatrix*}[r]
0 & 1 & 0 \\
-1 & 0 & 0 \\
0 & 0 & 0
\end{pmatrix*}, &
E_2 & = &
\begin{pmatrix*}[r]
0 & i & 0 \\
i & 0 & 0 \\
0 & 0 & 0
\end{pmatrix*},
\vspace{5pt}\\
E_3 & =
& \begin{pmatrix*}[r]
0 & 0 & 1 \\
0 & 0 & 0 \\
-1 & 0 & 0
\end{pmatrix*},
& 
E_4 & =
& \begin{pmatrix*}[r]
0 & 0 & i \\
0 & 0 & 0 \\
i & 0 & 0
\end{pmatrix*},
\vspace{5pt}\\
E_5 &=
& \begin{pmatrix*}[r]
0 & 0 & 0 \\
0 & 0 & 1 \\
0 & -1 & 0
\end{pmatrix*},
&
E_6 &=
& \begin{pmatrix*}[r]
0 & 0 & 0 \\
0 & 0 & i \\
0 & i & 0
\end{pmatrix*},
\vspace{5pt}\\
E_7 &=
& \begin{pmatrix*}[r]
-i/2 & 0 & 0 \\
0 & -i/2 & 0 \\
0 & 0 & i
\end{pmatrix*},
& 
E_8 & =
& \begin{pmatrix*}[r]
i & 0 & 0 \\
0 & -i & 0 \\
0 & 0 & 0
\end{pmatrix*}
.
\end{array}
\end{equation*}

We denote the dual basis of $E_1, \dots, E_8$ by $e_1, \dots, e_8$. 
The structure constants are
\\[10pt]
\begin{center}
$
{\renewcommand{\arraystretch}{2}
\begin{tabu}{r | r | r | r | r | r | r | r | r}
[ \cdot, \cdot ] & E_1 & E_2 & E_3 & E_4 & E_5 & E_6 & E_7 & E_8
\\ \hline
E_1 & 0 & 2 E_8 & -E_5 & -E_6 & E_3 & E_4 & 0 & -2E_2
\\ \hline
E_2 & -2 E_8 & 0 & E_6 & -E_5 & E_4 & -E_3 & 0 & 2E_1
\\ \hline
E_3 & E_5 & -E_6 & 0 & E_8-2E_7 & -E_1 & E_2 & \frac{3}{2}E_4 & -E_4
\\ \hline
E_4 & E_6 & E_5 & -E_8+2 E_7 & 0 & -E_2 & -E_1 & -\frac{3}{2} E_3 & E_3
\\ \hline
E_5 & -E_3 & -E_4 & E_1 & E_2 & 0 & -2 E_7 - E_8 & \frac{3}{2}E_6 & E_6
\\ \hline
E_6 & -E_4 & E_3 & -E_2 & E_1 & 2E_7 + E_8 & 0 & -\frac{3}{2}E_5 & - E_5
\\ \hline
E_7 & 0 & 0 & -\frac{3}{2}E_4 & \frac{3}{2} E_3 & -\frac{3}{2} E_6 & \frac{3}{2}E_5 & 0 & 0
\\ \hline
E_8 & 2E_2 & -2E_1 & E_4 & -E_3 & -E_6 & E_5 & 0 & 0
\end{tabu}
}
$
\end{center}

\begin{remark}
\label{equiv-classes-AW}
We now discuss various relations between the Aloff--Wallach spaces $N(k,l)$ for different pairs of integers. First, the subgroups $\U(1)_{k,l}$ and $\U(1)_{ak,al}$ coincide
and hence we can assume without loss of generality that the pair $(k,l)$ is coprime. Secondly, complex conjugation on $\SU(3)$ generates a group of outer automorphisms isomorphic to $\mathbb{Z}_2$ and maps $N(k,l)$ to $N(-k,-l)$. Finally, homogeneous spaces $G/H_1$ and $G/H_2$ are $G$-equivariantly diffeomorphic if the isotropy groups $H_1$ and $H_2$ are conjugate in $G$. The Weyl group of $\SU(3)$ is isomorphic to the symmetric group $S_3$ and conjugation by its elements permutes the triple $(k,l,-k-l)$ in formula (\ref{def-U_k,l}) accordingly. Therefore, it interchanges the subgroups $\U(1)_{k,l}, \U(1)_{l,-k-l}, \U(1)_{k,-k-l}$, etc., and partitions the set of Aloff--Wallach spaces into equivalence classes. 
\end{remark}

We have $\mathfrak{u}(1)_{1,-1} = \mathrm{span}\{E_8\}$
and the adjoint action of $\U(1)_{1,-1}$ maps the complement
$\mathfrak{m}=\mathrm{span}\{E_1, \dots, E_7 \}$ into itself.
Hence $T_{[\mathrm{Id}]}N(1,-1)$ can be identified with $\mathfrak{m}$ and an $\SU(3)$-invariant tensor field on $N(1,-1)$ corresponds to a tensor on $\mathfrak{m}$ which is left invariant by the adjoint action of $\U(1)_{1,-1}$. With respect to the basis $E_1, \dots, E_7$ the infinitesimal generator of the adjoint action is given by 
\begin{align*}
\mathrm{ad}(E_8)
=
\left(
\begin{array}{r@{}r@{}r@{}r}
\boxed{
\begin{array}{rr}
0 & -2 \\
2 & 0
\end{array}
}
\\ 
& 
\boxed{
\begin{array}{rr}
0 & -1 \\
1 & 0
\end{array}
}
\\
& & 
\boxed{
\begin{array}{rr}
0 & 1 \\
-1 & 0
\end{array}
}
\\
& & & 
\boxed{
\begin{array}{r}
0
\end{array}
}
\end{array}
\right).
\end{align*}
Hence $\mathfrak{m}$ splits into the
four irreducible $\U(1)$-modules
\begin{align*}
U_1 = \mathrm{span}\{E_1, E_2 \},\quad
U_2 = \mathrm{span}\{E_3, E_4 \},\quad
U_3 = \mathrm{span}\{E_5, E_6 \},\quad
U_4 = \mathrm{span}\{E_7\}.
\end{align*}
If we denote the irreducible representation of $\U(1)$ of weight $m$ by $\mathbb{C}_m$ we get 
\begin{align}
\label{decomp of tangent space of AW}
\mathfrak{m}
=
U_1\oplus U_2 \oplus U_3 \oplus U_4 
=
\mathbb{C}_2
\oplus \mathbb{C}_1
\oplus \mathbb{C}_{-1}
\oplus \mathbb{R},
\end{align}
Hence $\mathfrak{m}$ has two isotypical components. 
The equivalence classes of $N(1,-1)$ and $N(1,1)$ are the only equivalence classes with this property and therefore are called the \textit{exceptional} Aloff--Wallach spaces. The other Aloff--Wallach spaces are called \textit{generic}.

We fix the $\SU(3)$-invariant metric $q$ on $N(1,-1)$ 
for which $E_1, \dots , E_8$ is an orthonormal basis as a background metric
which allows us to consider any other $\SU(3)$-invariant metric $g$ on $N(1,-1)$ as an $\SU(3)$-invariant symmetric section of the endomorphism bundle. Because the submodules $U_2$ and $U_3$ are isomorphic, not every $\U(1)_{1,-1}$-invariant endomorphism of $\mathfrak{m}$ is diagonal. We can identify $U_2$ and $U_3$ with $\mathbb{C}$ by identifying $x E_3 + y E_4$ and $x E_5 + y E_6$ with $x+iy$, respectively. Over the real numbers the space of $\U(1)$-equivariant endomorphisms $\mathbb{C}_1 \rightarrow \mathbb{C}_{-1}$ is generated by $z\mapsto \bar{z}$ and $z\mapsto i \bar{z}$. With the above identifications this corresponds to $e_3 \otimes E_5 - e_4 \otimes E_6$ and $e_3 \otimes E_6 + e_4 \otimes E_5$, respectively. Hence any invariant symmetric endomorphism on $\mathfrak{m}$  with respect to the basis $E_1, \dots, E_7$ is of the form
\begin{gather}
a^2\, \mathrm{Id}_{U_1}+b^2\, \mathrm{Id}_{U_2}+c^2\, \mathrm{Id}_{U_3}+f^2\, \mathrm{Id}_{U_4}
\nonumber
\\
+
v (e_3 \otimes E_5 - e_4 \otimes E_6)
+
w (e_3 \otimes E_6 + e_4 \otimes E_5)
\label{sym-inv-end}
\\
+
v (e_5 \otimes E_3 - e_6 \otimes E_4)
+
w (e_5 \otimes E_4 + e_6 \otimes E_3)
.
\nonumber
\end{gather}
In particular, the space of $\SU(3)$-invariant metrics on $N(1,-1)$ is 6-dimensional.   

\begin{remark}
\label{remark-diagonal-metric}
In addition to the left multiplication of $G$ on $G/H$, there is another action given by conjugation with elements of the normaliser $\textrm{N}_G (H)=\{g\in G|\ gHg^{-1} = H \}$. In our case $\textrm{N}_{\SU(3)}(\U(1)_{1,-1})$ is the maximal torus of diagonal matrices in $\SU(3)$ isomorphic to $\U(1)^2$. We are particularly interested in the subgroup of the normaliser given by the embedding
\begin{align}
\label{extra-U(1)}
e^{i\theta}\mapsto
\begin{pmatrix}
e^{-i \theta} & 0 & 0\\
0 & e^{- i  \theta} & 0\\
0 & 0 & e^{i2\theta}
\end{pmatrix},
\end{align}
which is generated by $2E_7$.
The action of $E_7$ leaves the diagonal part of the endomorphism (\ref{sym-inv-end}) invariant but we have
\begin{align*}
&\mathrm{ad}(2E_7)(e_3 \otimes E_5 - e_4 \otimes E_6)
=
-6(e_3 \otimes E_6 + e_4 \otimes E_5),
\\
&\mathrm{ad}(2E_7)(e_3 \otimes E_6 + e_4 \otimes E_5)
=
6(e_3 \otimes E_5 - e_4 \otimes E_6).
\end{align*}
This has several consequences. First, Reidegeld \cite[p. 154]{reidegeld} concludes that in the non-diagonal case it suffices to consider 5 instead of 6 parameters. Secondly, 
any $\SU(3)$-invariant metric on $N(1,-1)$ with this additional $\U(1)$-symmetry
is diagonal, i.e. of the form
\begin{align}
\label{metric-on-orbit}
a^2(e_1^2 + e_2^2)
+
b^2(e_3^2 + e_4^2)
+
c^2(e_5^2 + e_6^2)
+
f^2 e_7^2.
\end{align}
We say that the metric is $\SU(3)\times \U(1)$-invariant.
The space of $\SU(3)\times \U(1)$-invariant metrics on $N(1,-1)$ is 4-dimensional.
\end{remark}

\begin{remark}
\label{rem-bundles}
Besides the Aloff--Wallach spaces, three further homogeneous spaces with a transitive action of $\SU(3)$ are relevant to us. 
$F_3 = \U(3)/\U(1)^3=\SU(3)/\U(1)^2$ is the manifold of complete flags in $\mathbb{C}^3$. 
If we embed $\SU(2)$ in $\SU(3)$ as
\begin{align*}
A \mapsto 
\begin{pmatrix}
A & 0
\\
0 & 1
\end{pmatrix},
\end{align*}
then
$S^5 = \SU(3)/\SU(2)$ is the 5-sphere and $\CP{2}=\SU(3)/(\SU(2)\times \U(1))$, where $\U(1)$ denotes the subgroup \eqref{extra-U(1)} of $\SU(3)$. Any Aloff--Wallach space $N(k,l)$ is a circle bundle over $F_3$. For example, the bundle structure of $N(1,-1)$ over $F_3$ is given by right multiplication with the circle \eqref{extra-U(1)} and thus is generated by $2E_7$. The flag manifold $F_3$ is the twistor space of $\CP{2}$, and in particular an $S^2$-bundle over $\CP{2}$. This leads to a fibration of each Aloff--Wallach space $N(k,l)$ over $\CP{2}$. As discussed in Remark \ref{equiv-classes-AW}, permutations of the triple $(k,l-k-l)$ lead to isomorphic Aloff--Wallach spaces. However, the fibration structure over $\CP{2}$ depends on the choice of a particular triple in an equivalent class. The fibres of $N(k,l)$ over $\CP{2}$ are given by the lens spaces $L(1,|k+l|)$ (for more details we refer to \cite[Section 4.1]{Gukov-Sparks}). Here for convenience $L(1,0)$ is defined to be $S^1\times S^2$.
The equivalence class of the Aloff--Wallach space $N(k,l)$ therefore gives rise to $L(1,|k|), L(1,|l|)$ and $L(1,|k+l|)$-bundles over $\CP{2}$. The equivalence classes of the exceptional Aloff--Wallach spaces $N(1,-1)$ and $N(1,1)$ give rise to two different bundle structures while the equivalence classes of the generic Aloff--Wallach spaces give rise to three different bundle structures. 
\end{remark}

\subsection{Spin(7)-structures with principal orbit $N(1,-1)$}

We now want to describe $\SU(3)\times \U(1)$-invariant cohomogeneity one Spin(7)-structures with principal orbit $N(1,-1)$, where the extra $\U(1)$-factor acts as described in Remark \ref{remark-diagonal-metric}.
We adopt the viewpoint from Section \ref{intro-coh1}
that outside the singular orbit a torsion-free cohomogeneity one Spin(7)-structure is a solution to the evolution equation (\ref{fundamental evo eqn}) in the space of co-closed invariant $\mathrm{G}_2$-structures on $N(1,-1)$. 
In Remark \ref{rem-bundles} we have explained that $N(1,-1)$ is a circle bundle over the flag manifold $F_3$. 
In the introduction we have noted the motivation for this work: 
torsion-free
ALC Spin(7)-structures with principal orbit $N(1,-1)$ collapsing to the $\SU(3)$-invariant Bryant--Salamon AC $\mathrm{G}_2$-metric on $\Lambda^2_{-}\CP{2}$, which is asymptotic to the cone over the homogeneous nearly K\"ahler structure on the flag manifold $F_3$.  
We want to describe $\SU(3)\times \U(1)$-invariant Spin(7)-structures with principal orbit $N(1,-1)$ in such a way that we can easily read off this fibration. 

We start with the base $F_3$. The circle bundle structure of $N(1,-1)$ over $F_3$ is generated by $2E_7$.
Therefore  $T_{[\mathrm{Id}]}F_3\cong U_1 \oplus U_2 \oplus U_3=\mathrm{span}\{E_1, \cdots, E_6 \}$. 
We denote $e_{i_1}\wedge \cdots \wedge e_{i_k}$ by $e_{i_1 \dots i_k}$.
The $\SU(3)$-invariant nearly K\"ahler structure on $F_3$ is given by (see \cite[Section 6]{NK})
\begin{subequations}
\label{NK structure on F3}
\begin{align}
\omega_{0}
&=
e_{12} + e_{43} + e_{56},
\\
\Omega_0
&=
e_{136}+e_{246}+e_{235}-e_{145}
+
i
(
e_{236}-e_{146}-e_{135}-e_{245}
).
\end{align}
\end{subequations}
To determine the space of $\SU(3)$-invariant $\mathrm{G}_2$-structures on $N(1,-1)$ we need to compute the other invariant 3-forms. In the course of this computation we need the following 

\begin{lemma}
\label{reps-of-U(1)}
Let $\mathbb{C}_k=\mathrm{Span}\{v_1, v_2\}$
be an oriented $\U(1)$-module of weight k and 
$\mathbb{C}_l=\mathrm{Span}\{v'_1, v'_2\}$
be an oriented $\U(1)$-module of weight l.
Then as oriented $\U(1)$-modules we have
\begin{align*}
\mathbb{C}_k \otimes \mathbb{C}_l
=
\mathbb{C}_{k+l} \oplus \mathbb{C}_{k-l},
\end{align*}
where
\begin{align*}
\mathbb{C}_{k+l}
=
\mathrm{Span}
\{
v_1 \otimes v'_1 - v_2\otimes v'_2,
v_1\otimes v'_2 + v_2 \otimes v'_1 
 \},
 \\
\mathbb{C}_{k-l}
=
\mathrm{Span}
\{
v_1 \otimes v'_2 - v_2\otimes v'_1,
v_1\otimes v'_1 + v_2 \otimes v'_2 
 \}
 .
\end{align*}
\end{lemma}

\begin{lemma}
\label{lemma-invariant forms}
\begin{compactenum}[(i)]
\item The space of $\SU(3)$-invariant 1-forms on $N(1,-1)$ is spanned by $e_7$.
\item The space of $\SU(3)$-invariant 2-forms on $N(1,-1)$ is five dimensional and spanned by 
\begin{gather*}
e_{12},\, e_{34},\, e_{56},\, e_{35}-e_{46},\, e_{36}+e_{45}.
\end{gather*}
\item The space of $\SU(3)$-invariant 3-forms on $N(1,-1)$ is seven dimensional and spanned by 
\begin{gather*}
e_{127},\, e_{347},\, e_{567},
\,
e_{357}-e_{467},
\,
e_{367}+e_{457},
\\
\mathrm{Re}\ \Omega_0 = e_{136}-e_{145}+e_{235}+e_{246},
\,
\mathrm{Im}\ \Omega_0 = -e_{146}-e_{135}+e_{236}-e_{245}.
\end{gather*}
\end{compactenum}
\end{lemma}
\begin{proof}
(i) follows immediately from (\ref{decomp of tangent space of AW}).

\noindent
(ii) As $\U(1)$-modules we have 
\begin{align*}
U_1^* \cong U_1\cong \mathbb{C}_2,
\quad 
U_2^* \cong U_2 \cong \mathbb{C}_1,
\quad
U_3^*\cong U_3\cong  \mathbb{C}_{-1},
\quad
U_4^* \cong U_4\cong \mathbb{R}.
\end{align*}
Using Lemma \ref{reps-of-U(1)} we compute the invariant 2-forms:
\begin{align*}
\Lambda^2 \mathfrak{m}^*
&=
\Lambda^2 U_1
\oplus
\Lambda^2 U_2
\oplus
\Lambda^2 U_3
\oplus
(U_1
\oplus
U_2
\oplus
U_3)
\otimes U_4
\oplus
(U_1 \otimes U_2)
\oplus
(U_1 \otimes U_3) 
\oplus
(U_2 \otimes U_3)
\\
&\cong
\mathbb{R}^3
\oplus
\mathbb{C}_2
\oplus
\mathbb{C}_1
\oplus
\mathbb{C}_{-1}
\oplus
(\mathbb{C}_3 \oplus \mathbb{C}_1)
\oplus
(\mathbb{C}_3 \oplus \mathbb{C}_1)
\oplus
(\mathbb{C}_2 \oplus \mathbb{R}^2).
\end{align*}
If we write $U_2^*=\mathrm{span}\{e_3,e_4\}$ and $U_3^*=\mathrm{span}\{e_5,e_6\}$ as oriented $\U(1)$-modules of weight 1 and -1, respectively,
then the trivial $\mathbb{R}^2\subset U_2^* \otimes U_3^*$
is spanned by
\begin{align*}
e_{35}-e_{46}, \quad e_{36}+e_{45}.
\end{align*}
Hence the space of invariant 2-forms is 5-dimensional and spanned by the claimed forms.

\noindent
(iii) The space of 3-forms decomposes as
\begin{align*}
\Lambda^3 \mathfrak{m}^*
=
&\Lambda^2 U_1 \otimes(U_2 \oplus U_3 \oplus \mathbb{R})
\oplus
\Lambda^2 U_2 \otimes(U_1 \oplus U_3 \oplus \mathbb{R})
\oplus
\Lambda^2 U_3 \otimes(U_1 \oplus U_2 \oplus \mathbb{R})
\\
&\oplus
(U_1 \otimes U_2 \otimes U_3)
\oplus
(U_1 \otimes U_2 \otimes \mathbb{R})
\oplus
(U_1 \otimes U_3 \otimes \mathbb{R})
\oplus
(U_2 \otimes U_3 \otimes \mathbb{R}).
\end{align*}
We have
\begin{align*}
U_1 \otimes U_2 \otimes U_3
=
\mathbb{C}_2 \otimes \mathbb{C}_1 \otimes \mathbb{C}_{-1}
=
\mathbb{C}_4 \oplus 2\mathbb{C}_2 \oplus \mathbb{R}^2.
\end{align*}
The $\mathbb{C}_1$ part in $U_1\otimes U_2$
is spanned by
\begin{align*}
e_{14}-e_{23},\quad e_{13}+e_{24}.
\end{align*}
Hence the invariant part of $U_1\otimes U_2 \otimes U_3$ is spanned by
\begin{align*}
(e_{14}-e_{23})\wedge e_5 -(e_{13}+e_{24})\wedge e_6,
\quad
(e_{14}-e_{23})\wedge e_6 +(e_{13}+e_{24})\wedge e_5.
\end{align*}
We conclude that the space of invariant 3-forms is 7-dimensional 
and spanned by the claimed forms.
\end{proof}

Using the Maurer--Cartan equation and the structural constants we can compute the exterior derivatives of some of the invariant forms computed in Lemma \ref{lemma-invariant forms}.

\begin{lemma}
\label{lemma-exterior-derivative}
We have
\begin{gather*}
de_7 = -2 e_{43} + 2 e_{56},
\\
de_{12} =
de_{43} =
de_{56} =\ \mathrm{Re}\ \Omega_0,
\\
de_{127} =\ \mathrm{Re}\ \Omega_0\wedge e_7 -2 e_{1243} + 2 e_{1256},
\\
de_{437} =\ \mathrm{Re}\ \Omega_0\wedge e_7 + 2 e_{4356},
\\
de_{567} =\ \mathrm{Re}\ \Omega_0\wedge e_7 - 2 e_{4356},
\\
d\mathrm{Im}\ \Omega_0 = -2\omega_0^2,
\\
d(e_{357}-e_{467}) = d(e_{367}+e_{457}) = 0.
\end{gather*}
\end{lemma}

Now we are ready to describe $\SU(3)\times \U(1)$-invariant Spin(7)-structures on $(0,\infty)\times N(1,-1)$ in a way such that the asymptotic behaviour can be conveniently read off the coefficient functions.
Starting with the homogeneous nearly K\"ahler structure (\ref{NK structure on F3}) on $F_3$, we can scale $U_1, U_2, U_3$ respectively 
by non-zero $a, b, c $ to get the invariant $\SU(3)$-structure
\begin{align*}
\omega &= a^2 e_{12} + b^2 e_{43} + c^2 e_{56},
\\
\Omega &= abc\ \Omega_0.
\end{align*}
On $(0,\infty)\times F_3$ we evolve such $\SU(3)$-structures to get the $G_2$-structure
\begin{align*}
\tilde{\varphi}
&=
dt\wedge \omega + \mathrm{Re}\ \Omega
=
a^2 dt\wedge e_{12}
+
b^2 dt\wedge e_{43} 
+
c^2 dt\wedge e_{56} 
+ 
abc\ \mathrm{Re}\ \Omega_0,
\\
*\tilde{\varphi}
&=
\frac{1}{2}\omega^2-dt\wedge \mathrm{Im}\ \Omega
=
a^2b^2 e_{1243}+ b^2 c^2 e_{4356} + c^2 a^2 e_{1256}
-
abc\  dt\wedge\mathrm{Im}\ \Omega_0.
\end{align*}
If we now consider $(0,\infty)\times N(1,-1)$ as a circle bundle over $(0,\infty)\times F_3$, this $G_2$-structure 
together with the multiple $-f e_7$ of the invariant connection gives the Spin(7)-structure
\begin{align}
\psi
=&
(-f e_7)\wedge \tilde{\varphi} +*\tilde{\varphi}
\nonumber
\\
=&
(-f e_7)\wedge 
\big(
a^2 dt\wedge e_{12}
+
b^2 dt\wedge e_{43} 
+
c^2 dt\wedge e_{56} 
+ 
abc\ \mathrm{Re}\ \Omega_0
\big)
\nonumber
\\
&
+
\big(
a^2b^2 e_{1243}+ b^2 c^2 e_{4356} + c^2 a^2 e_{1256}
-
abc\  dt\wedge\mathrm{Im}\ \Omega_0
\big)
\nonumber
\\
=&
a^2 f dt\wedge e_{127}
+
b^2 f dt\wedge e_{437}
+
c^2 f dt\wedge e_{567}
-abc\ dt\wedge\mathrm{Im}\ \Omega_0
\nonumber
\\
&
+
abcf\ 
\mathrm{Re}\ \Omega_0\wedge e_7
+
a^2b^2 e_{1243}+ b^2 c^2 e_{4356} + c^2 a^2 e_{1256}.
\label{Spin(7)-structure}
\end{align}
By the formulas \eqref{G2-reduction-SU(3)} and \eqref{Spin(7)-reduction-G2} the Spin(7)-structure (\ref{Spin(7)-structure})
induces the metric
\begin{align}
\label{Spin(7)-metric}
g = dt^2 + a^2 (e_1^2 + e_2^2) + b^2 (e_3^2 + e_4^2) + c^2 (e_5^2 + e_6^2) + f^2 e_7^2.
\end{align}
\begin{remark}
\label{remark-ALC asymp from coeff}
As promised the choice of parameters $a, b, c, f$ easily allows to read off the asymptotic behaviour.
Because the nearly K\"ahler structure on $F_3$ is given by $a = b = c = 1$ and the coefficient $f$ describes the length of the circle fibres of the circle bundle $(0,\infty)\times N(1,-1) \rightarrow (0,\infty)\times F_3$, $\psi$ is an ALC Spin(7)-structure asymptotic to a circle bundle with fibre length $\ell$ over the $\mathrm{G}_2$-cone over the homogeneous nearly K\"ahler structure on $F_3$ if
\begin{align*}
a(t)/t \rightarrow 1,\quad b(t)/t\rightarrow 1, \quad c(t)/t\rightarrow 1,
\quad f(t) \rightarrow \ell
\quad
\text{as}\ t\rightarrow \infty.
\end{align*}
\end{remark}

While the above construction of the Spin(7)-structure is helpful in reading off the asymptotic behaviour, it is not compatible with the viewpoint from Section \ref{intro-coh1} that cohomogeneity one Spin(7)-metrics correspond to an evolution of $\mathrm{G}_2$-structures. However, alternatively we can consider $N(1,-1)$ as a circle bundle
over $F_3$ now equipped with the rotated $\SU(3)$-structure $(\omega,\widetilde{\Omega})=(\omega,i\Omega)$.
Then on $N(1,-1)$ we get the $G_2$-structure
\begin{align*}
\varphi
=&
(f e_7)\wedge \omega + \mathrm{Re}\ \widetilde{\Omega}
\\
=&
f \omega\wedge e_7 -abc\ \mathrm{Im}\ \Omega_0
\\
=&
a^2 f e_{127} + b^2 f e_{437} + c^2 f e_{567}
-
abc\ \mathrm{Im}\ \Omega_0
,
\\
*\varphi
=&
\frac{1}{2}\omega^2-(f e_7)\wedge\mathrm{Im}\ \widetilde{\Omega}
\\
=&
a^2 b^2 e_{1243} + b^2c^2 e_{4356} + c^2a^2 e_{1256}
+ 
abcf\ \mathrm{Re}\ \Omega_0\wedge e_7.
\end{align*}
This $G_2$-structure induces on $(0,\infty)\times N(1,-1)$
the Spin(7)-structure
\begin{align*}
\psi
=
dt\wedge \varphi + *\varphi,
\end{align*}
which coincides with (\ref{Spin(7)-structure}).

\begin{remark}
\label{remark-inv-spin(7)}
In Remark \ref{remark-diagonal-metric} we showed that any $\SU(3)\times \U(1)$-invariant metric on $N(1,-1)$ is purely diagonal. 
Furthermore, a direct computation shows that the additional $\U(1)$-action also preserves the $\mathrm{G}_2$-structure $\varphi$. Therefore, we are really studying $\SU(3)\times \U(1)$-invariant Spin(7)-structures.
\end{remark}

The next Lemma shows that the static part of the torsion-free condition (\ref{tf condition-static eqn}), i.e. that $\varphi$ is coclosed, is always satisfied. 
\begin{lemma}
\label{lemma-static eq}
The $\mathrm{G}_2$-structure $\varphi$ is coclosed.
\end{lemma}
\begin{proof}
Using Lemma \ref{lemma-exterior-derivative} we get
\begin{align*}
d*\varphi
=&\
a^2b^2 \mathrm{Re}\ \Omega_0\wedge(e_{12}+e_{43})
+
b^2c^2 \mathrm{Re}\ \Omega_0\wedge(e_{43}+e_{56})
+
c^2a^2 \mathrm{Re}\ \Omega_0\wedge(e_{56}+e_{12})
\\
&+
abcf (d\mathrm{Re}\ \Omega_0\wedge e_7-\mathrm{Re}\ \Omega_0\wedge(-2e_{43}+2e_{56})).
\end{align*}
The result follows because 
\begin{align*}
\mathrm{Re}\ \Omega_0\wedge e_{12}
=
\mathrm{Re}\ \Omega_0\wedge e_{43}
=
\mathrm{Re}\ \Omega_0\wedge e_{56}
=
0
\end{align*}
and $\mathrm{Re}\ \Omega_0$ is closed by the nearly K\"ahler condition (\ref{NK condition}).
\end{proof}

\begin{remark}
\label{conn-comp-Spin(7)}
With (\ref{Spin(7)-structure}) we have constructed one $\SU(3)\times \U(1)$-invariant Spin(7)-structure which induces the metric (\ref{Spin(7)-metric}). To see if there are others, we can ask equivalently what $\SU(3)\times \U(1)$-invariant $\mathrm{G}_2$-structures other than $\varphi$ induce the metric (\ref{metric-on-orbit}) on $N(1,-1)$. Any other $\SU(3)$-invariant $\mathrm{G}_2$-structure on $N(1,-1)$ can be obtained from $\varphi$ by pulling back with an $\SU(3)$-invariant  diffeomorphism, i.e. a $\U(1)_{1,-1}$-invariant isomorphism of $\mathfrak{m}$. These are given by the normaliser $\mathrm{N}_{\mathrm{GL}(7,\mathbb{R})}\U(1)_{1,-1}$. Those which give rise to the same metric are given by $\mathrm{N}_{\SO(7)}\U(1)_{1,-1}$. As $\varphi$ is preserved exactly by $\mathrm{N}_{\mathrm{G}_2}\U(1)_{1,-1}$ we see that the set of all $\SU(3)$-invariant $\mathrm{G}_2$-structures on $N(1,-1)$ which induce the metric (\ref{metric-on-orbit}) is parametrised by $\mathrm{N}_{\SO(7)}\U(1)_{1,-1}/\mathrm{N}_{\mathrm{G}_2}\U(1)_{1,-1}$. Reidegeld \cite[(42) on p. 22]{Reidegeld-paper} has shown that the connected component of the identity is isomorphic to $\U(1)$. Furthermore, he has shown that $\psi$ is up to discrete symmetries the only $\SU(3)\times \U(1)$-invariant Spin(7)-structure inducing the metric $g$ which can be torsion-free \cite[Theorem 4.4 (2)]{Reidegeld-paper}. The reason is that the other invariant $\mathrm{G}_2$-structures in the connected component of $\varphi$ are not coclosed, i.e. fail to solve the static condition (\ref{tf condition-static eqn}). 
\end{remark}

The evolution equation (\ref{fundamental evo eqn}) given by
$
d\varphi = \partial_{t} *\varphi
$
is equivalent to an ODE system for the coefficient functions $a, b, c, f$.

\begin{proposition}
The Spin(7)-structure \eqref{Spin(7)-structure} on $I\times N(1,-1)$, where $I \subset \mathbb{R}_t$ is some interval, is torsion-free if and only if $(a,b,c,f)$ is a solution of the ODE system
\begin{subequations}
\label{ODE system}
\begin{align}
\frac{\dot{a}}{a} &= \frac{b^2+c^2-a^2}{abc},
\label{ODE system - a}
\\
\frac{\dot{b}}{b}
&=
\frac{c^2+a^2-b^2}{abc}- \frac{f}{b^2},
\label{ODE system - b}
\\
\frac{\dot{c}}{c}
&=
\frac{a^2+b^2-c^2}{abc}+ \frac{f}{c^2},
\label{ODE system - c}
\\
\frac{\dot{f}}{f}
&=
\frac{f}{b^2}-\frac{f}{c^2}.
\label{ODE system - f}
\end{align}
\end{subequations}
The holonomy of the associated metric is all of Spin(7).
\end{proposition}
\begin{proof}
$\psi$ is torsion-free if and only if $\varphi$ solves the system \eqref{tf condition}. By Lemma \ref{lemma-static eq} the static equation \eqref{tf condition-static eqn} is always satisfied. The evolution equation \eqref{tf condition-evo eqn} is equivalent to a system of ODEs, which we now derive using Lemma \ref{lemma-exterior-derivative}.
\begin{align*}
d\varphi
=&
a^2 f (\mathrm{Re}\ \Omega_0\wedge e_7 -2 e_{1243} +2 e_{1256})
+
b^2 f (\mathrm{Re}\ \Omega_0\wedge e_7 +2 e_{4356})
\\&
+
c^2 f (\mathrm{Re}\ \Omega_0\wedge e_7 -2 e_{4356})
+
2abc\ \omega_0^2
\\
=&
(a^2+b^2+c^2)f\ \mathrm{Re}\ \Omega_0\wedge e_7
\\
&+
(-2 a^2 f +4abc) e_{1243}
+
(2 a^2 f + 4abc) e_{1256}
+
(2 b^2f-2 c^2f +4abc) e_{4356}.
\end{align*}
Equating this with $\partial_{t}*\varphi$ leads to the system
\begin{align*}
\partial_{t}(a^2b^2) &= - 2 a^2 f +4abc,
\\
\partial_{t}(b^2c^2) &= 2 b^2f-2 c^2f +4abc,
\\
\partial_{t}(c^2 a^2) &= 2 a^2 f + 4abc,
\\
\partial_{t}(abcf) &= (a^2+b^2+c^2)f.
\end{align*}
Denoting differentiation with respect to $t$ by a dot, we can simplify the above system to get
\begin{align*}
\frac{\dot{a}}{a}+ \frac{\dot{b}}{b}
&=
-\frac{f}{b^2}+2\frac{c^2}{abc}
\\
\frac{\dot{b}}{b}+ \frac{\dot{c}}{c}
&=
\frac{f}{c^2}-\frac{f}{b^2}+2\frac{a^2}{abc}
\\
\frac{\dot{c}}{c}+ \frac{\dot{a}}{a}
&=
\frac{f}{c^2}+2\frac{b^2}{abc}
\\
\frac{\dot{a}}{a}+ \frac{\dot{b}}{b}
+
\frac{\dot{c}}{c}+\frac{\dot{f}}{f}
&=
\frac{a^2+b^2+c^2}{abc}.
\end{align*}
This finally gives
\eqref{ODE system}.
The statement about the holonomy group follows from \cite[Theorem 4.4]{Reidegeld-paper}
\end{proof}

\begin{remark}
\label{rem-reduced system}
The system (\ref{ODE system}) is compatible with $f\equiv 0, b\equiv c$.
It reduces to
\begin{subequations}
\label{reduced BS system}
\begin{align}
\dot{a}
&=
2-\frac{a^2}{b^2},
\\
\dot{b}
&=
\frac{a}{b}.
\end{align} 
\end{subequations}
If we introduce the parameter $r$ such that $b(r)=r$ we get the general solution
\begin{align*}
a(r)=r
\left(
1+
\frac{C}{r^4}
\right)^{\frac{1}{2}},
\end{align*}
where $C$ is a constant of integration. Up to the scale of $(a,b)$ there are three cases corresponding to $C=-1,0,1$. The case $C=0$ gives the $\mathrm{G}_2$ holonomy cone $C(F_3)$ over the homogeneous nearly K\"ahler structure on $F_3$. $C=-1$ with the constraint $r \geq 1$ gives the complete Bryant--Salamon metric on $\Lambda_{-}^2\CP{2}$, which is asymptotic to $C(F_3)$. The case $C=1$ gives a $\mathrm{G}_2$ metric which is singular as $r\rightarrow 0$ and asymptotic to $C(F_3)$ as $r\rightarrow\infty$. After interchanging $a$ and $c$ this singular space is the collapsed limit of the families $\Omega^z_{\kappa}$ in Theorem \ref{thm-sing}.
\end{remark}
\begin{remark}
\label{remark- AC asymp from coeff}
As explained in remark \ref{remark-ALC asymp from coeff} ALC asymptotics can be easily read off from the coefficient functions $a, b, c, f$. The same is true for an AC Spin(7)-structure asymptotic to the cone over the diagonal $\SU(3)$-invariant nearly parallel $\mathrm{G}_2$-structure on $N(1,-1)$. Substituting the coefficients $a(t) = a_c\, t, b(t) = b_c\, t, c(t) = c_c\, t, f(t) =  f_c\, t$ of the conical Spin(7)-structure in the system (\ref{ODE system}) gives 
\begin{align*}
a_c= \frac{2}{\sqrt{5}}\approx 0.89,
\quad
b_c = \sqrt{\frac{2}{15}(5-\sqrt{5})}\approx 0.61,
\\
c_c = \sqrt{\frac{2}{15}(5+\sqrt{5})}\approx 0.98,
\quad
f_c = \frac{4}{3\sqrt{5}}\approx 0.60.
\end{align*}
\end{remark}

\begin{remark}
\label{princ-orbit-N10}
More generally, $\SU(3)$-invariant torsion-free Spin(7)-structures with principal orbit a generic Aloff--Wallach space $N(k,l)$ are characterised by the ODE system 
\begin{subequations}
\label{general-ODEsystem}
\begin{align}
\frac{\dot{a}}{a} &= \frac{b^2+c^2-a^2}{abc}+\frac{m}{\Delta} \frac{f}{a^2},
\\
\frac{\dot{b}}{b}
&=
\frac{c^2+a^2-b^2}{abc}+\frac{l}{\Delta} \frac{f}{b^2},
\\
\frac{\dot{c}}{c}
&=
\frac{a^2+b^2-c^2}{abc}+\frac{k}{\Delta} \frac{f}{c^2},
\\
\frac{\dot{f}}{f}
&=
-\frac{m}{\Delta} \frac{f}{a^2}-\frac{l}{\Delta}
\frac{f}{b^2}-\frac{k}{\Delta}\frac{f}{c^2}.
\end{align}
\end{subequations}
Here $m = -k -l$ and $\Delta = k^2 + k l + l^2$.
Besides $N(1,-1)$ we are also interested in the principal orbit $N(1,0)$, which is equivariantly diffeomorphic to $N(1,-1)$. Note that the system (\ref{general-ODEsystem}) for $(k,l,m) = (1,0,-1)$ coincides with the system for $(k,l,m)=(1,-1,0)$ after swapping $a$ and $b$. For us it will be convenient to consider cohomogeneity one torsion-free Spin(7)-structures with principal orbit $N(1,0)$ as solutions of the system (\ref{ODE system}) after exchanging the initial conditions for $a$ and $b$.   
\end{remark}

\subsection{Preservation laws and a coordinate change on projective space}
\label{section-coord change}

To understand the long-time behaviour of local solutions of the system (\ref{ODE system}) it is crucial to understand preserved orderings of the functions $a,b,c$ and $f$. The following Lemma is an elementary yet important observation.

\begin{lemma}
\label{first observations}
Assume that a (local) solution $(a,b,c,f)$ of the system (\ref{ODE system}), where $a, b, c, f$ are positive functions, satisfies both
\begin{compactenum}[(i)]
\item
$b < c$\, \text{and}
\item 
$ a < c$
\end{compactenum}
at some time $t_0$.
This set of conditions is preserved forward as long as the solution exists,
and $f$ is strictly monotone increasing from then onwards.
\end{lemma}
\begin{proof}
As long as the solution exists all functions stay positive.
\\
(i) Assume $b(t_1)=c(t_1)$ for some $t_1 > t_0$. Then at time $t_1$
\begin{align*}
\dot{b}&=\frac{a}{b}-\frac{f}{b},
\\
\dot{c}&=\frac{a}{b}+\frac{f}{b}.
\end{align*}
Because $f > 0$ we get $\dot{b}(t_1) < \dot{c}(t_1)$,
which is a contradiction.
\\ 
(ii)
Assume $a(t_1) = c(t_1)$ for some $t_1 > t_0$. 
Then at time $t_1$
\begin{align*}
\dot{c} = \frac{b}{a} +  \frac{f}{c}
> 
\frac{b}{a} = \dot{a}.
\end{align*}
The monotonicity of $f$ is a direct consequence of (i) as 
\begin{align*}
\dot{f}=\frac{f^2}{b^2}-\frac{f^2}{c^2}.
\end{align*}
\end{proof}

The previous Lemma suggests that the quotients $a/c$ and $b/c$ are well-behaved. 
Because the right-hand side of the ODE system (\ref{ODE system})
is homogeneous we can consider the system in the projective coordinates
\begin{align}
\label{proj-coord}
A = \frac{a}{c},
\quad
B = \frac{b}{c},
\quad
F = \frac{f}{c}.
\end{align}
A similar use of projective coordinates was made by
Atiyah--Hitchin \cite[Chapter 9]{AH}.
In the following we derive the evolution equations in these coordinates.
\begin{align}
&\frac{d}{dt}
\log a 
- 
\frac{d}{dt}
\log c
=
2 \frac{c^2-a^2}{abc}-\frac{f}{c^2},
\nonumber
\\
&
\frac{d}{dt}\frac{a}{c}
=
\frac{2}{b}
\left(
1
-
\left(
\frac{a}{c}
\right)^2
\right)
-\frac{af}{c^3},
\label{ODE system - proj coord 1 - prelim}
\\[15pt]
&
\frac{d}{dt}\log b
-
\frac{d}{dt}\log c
=
2\frac{c^2-b^2}{abc}
-
f
\left(
\frac{1}{b^2}+\frac{1}{c^2}
\right),
\nonumber
\\
&
\frac{d}{dt}\frac{b}{c}
=
\frac{2}{a}
\left(
1
-
\left(
\frac{b}{c}
\right)^2
\right)
-
\frac{f}{bc}
\left(
1+
\left(
\frac{b}{c}
\right)^2
\right),
\label{ODE system - proj coord 2 - prelim}
\\[15pt]
&
\frac{d}{dt}\log f 
-
\frac{d}{dt}\log c
=
\frac{f}{b^2}-2\frac{f}{c^2}
-
\frac{a^2+b^2-c^2}{abc},
\nonumber
\\
&
\frac{d}{dt}
\frac{f}{c}
=
\frac{f^2}{b^2 c}
-2
\frac{f^2}{c^3}
+
\frac{f}{c}
\frac{c^2-a^2-b^2}{abc}.
\label{ODE system - proj coord 3 - prelim}
\end{align}

Changing the parameter by $dt=\frac{ab}{c}ds$
(\ref{ODE system - proj coord 1 - prelim})-(\ref{ODE system - proj coord 3 - prelim}) becomes
\begin{align*}
\frac{d}{ds} A 
&=
\frac{d}{ds}
\frac{a}{c}
=
2\frac{a}{c}
\left(
1
-
\left(
\frac{a}{c}
\right)^2
\right)
-\frac{a^2bf}{c^4}
\\
&=
2A(1-A^2)-A^2BF,
\\[15pt]
\frac{d}{ds} B 
&=
\frac{d}{ds}
\frac{b}{c}
=
2 \frac{b}{c}
\left(
1
-
\left(
\frac{b}{c}
\right)^2
\right)
-
\frac{af}{c^2}
\left(
1+
\left(
\frac{b}{c}
\right)^2
\right)
\\
&=
2B(1-B^2)-AF(1+B^2),
\\[15pt]
\frac{d}{ds} F 
&=
\frac{d}{ds}
\frac{f}{c}
=
\frac{a f^2}{b c^2}
-2
\frac{abf^2}{c^4}
+
\frac{f}{c}
\frac{c^2-a^2-b^2}{c^2}
\\
&=
\frac{AF^2}{B}
-
2ABF^2
+F(1-A^2-B^2).
\end{align*}

To sum up, if we denote differentiation with respect to $s$ by a dot, then the system (\ref{ODE system}) takes the form
\begin{subequations}
\label{ODE system in proj coord}
\begin{align}
\dot{A}
&=
A
\left(2-2 A^2-ABF
\right)
,
\\
\dot{B}
&=
B
\left(
2-2 B^2-ABF-\frac{AF}{B}
\right)
,
\\
\dot{F}
&=
F
\left(
1-A^2-B^2-2ABF+\frac{AF}{B}
\right)
.
\end{align}
\end{subequations}

The main difficulty in the analysis of the ODE system (\ref{ODE system}) is that apart from monotonicity under the conditions (i) and (ii) in Lemma \ref{first observations} nothing can be said about the behaviour of $f$ in relation to any of the other functions. In particular, it is of concern that $f$ blows up in finite time. The lack of control of $f$ is reflected by the fact that for the system (\ref{ODE system in proj coord}) no bounds can be derived for $F$. A key observation is that the controlled quantities $a/c$ and $b/c$ dominate the ill-behaved quantity $f/c$. 
To be more precise, set
\begin{align}
X = A^2, 
\quad
Y = B^2,
\quad
Z = ABF.
\end{align}
Still denoting differentiation with respect to the variable $s$ by a dot, the ODE system takes the form
\begin{subequations}
\label{ODE system in good coord}
\begin{align}
\dot{X} 
&= 
2X(2-2X-Z),
\\
\dot{Y}
&=
4Y-4Y^2-2YZ-2Z,
\\
\dot{Z}
&=
Z(5-3X-3Y-4Z).
\end{align}
\end{subequations}

\begin{remark}
\label{associated solutions}
Let $(a(t),b(t),c(t),f(t))$ be positive functions which solve the system (\ref{ODE system}) for $t$ in the interval $(T_1, T_2)$. Then there exists a corresponding solution $(X(s),\allowbreak Y(s),\allowbreak Z(s))$ of (\ref{ODE system in good coord}) defined on the interval $(S_1,S_2)$, where $S_1\in \{-\infty\} \cup \mathbb{R}$, $S_2 \in \mathbb{R}\cup \{\infty\}$.
After choosing $s(t_0)$ arbitrarily for some $t_0\in (T_1, T_2)$, because of $dt=\frac{ab}{c}ds$ the $s$-parameter is given by
\begin{align*}
s(t) = \int_{t_0}^t \frac{c(\tilde{t})}{a(\tilde{t})b(\tilde{t})} d\tilde{t} + s(t_0).
\end{align*}
This is well-defined because $a,b,c$ are positive functions.
We will say that the solution 
\\
$(X(s),Y(s),Z(s))$ is \textit{associated} to
$(a(t), b(t), c(t), f(t))$.
\end{remark}

All of the information on $f$ is contained in $Z$. We are finally able to control this quantity.
\begin{lemma}
\label{bounds in xyz coords}
Assume that  a (local) solution $(X,Y,Z)$ of the system (\ref{ODE system in good coord}) satisfies all of the three conditions
\begin{compactenum}[(i)]
\item
$0 < X < 1$,
\item 
$ Y < 1$,
\item
$ 0 < Z < \kappa, \quad \kappa \geq \frac{5}{4}$,
\end{compactenum}
at some time $s_0$.
Then this set of conditions is preserved forward as long as $Y>0$.
\end{lemma}
\begin{proof}
$ 0 < X, Z$ is preserved as the system (\ref{ODE system in good coord}) is compatible with $X\equiv 0$ and $Z\equiv 0$.
\\
(i) Assume $X(s_1)=1$ for some $s_1 > s_0$. Then at time $s_1$
\begin{align*}
\dot{X} = -2Z < 0.
\end{align*}
(ii) Assume $Y(s_1) =1$ for some $s_1 > s_0$. Then at time $s_1$ 
\begin{align*}
\dot{Y}= -4 Z < 0.
\end{align*}
(iii) Assume $Z(s_1)= \kappa$ with $\kappa \geq \frac{5}{4}$ for some $s_1 > s_0$. Then at time $s_1$ 
\begin{align*}
\dot{Z} = 
4Z (5/4-Z)-3Z(X+Y)
\leq -3Z(X+Y) < 0.
\end{align*}
All cases lead to a contradiction.
\end{proof}
Besides controlling $f$ we also got rid of all singularities on the right-hand side of the ODE system. This means that a local solution $(X,Y,Z)$ can only develop a singularity by shooting off to infinity in finite time. If we start with the conditions in Lemma \ref{bounds in xyz coords} 
the solution is contained in a compact cube until it hits the hypersurface $Y=0$. If $(X,Y,Z)$ is a solution associated with
a solution $(a,b,c,f)$ of the system (\ref{ODE system}),
$Y=0$ implies $b=0$, i.e. the original solution already develops a singularity at $Y=0$.
To sum up, for all solutions of $(\ref{ODE system})$ that we are interested in 
we have enough preservation laws such that the long-time behaviour is encoded only in the ratio $b/c$. More precisely we get

\begin{lemma}
\label{lemma-reduction to proj system}
Let $(a,b,c,f)$ be a (local) solution of the system (\ref{ODE system}), where $a,b,c,f$ are positive functions satisfying $a, b < c$. If for the associated solution $(X,Y,Z)$ of the system (\ref{ODE system in good coord}) given by Remark \ref{associated solutions}  the function $Y$ stays bounded away from zero, then the solution $(X,Y,Z)$ is forward complete, i.e. it exists for all large $s$. Moreover, $(a,b,c,f)$ itself is forward complete, i.e. it exists for all large $t$.
\end{lemma}
\begin{proof}
Because $a,b,c,f$ are positive and we have $ a, b < c$, the conditions of Lemma \ref{bounds in xyz coords} are satisfied for some $\kappa$. As they are preserved and we assume that $Y$ stays bounded away from zero the solution $(X,Y,Z)$ is contained in a compact region and is therefore forward complete and positive for all $s$. 
To obtain $(a,b,c,f)$ from $(X,Y,Z)$ we need to make one more integration. With $ a \sqrt{Y}\ ds=dt$ we can reformulate the evolution equation (\ref{ODE system - a}) for $a$ as
\begin{align*}
\frac{d}{ds}\log a
=
\frac{1}{a}\frac{da}{ds}
= \sqrt{Y} \frac{da}{dt}
= Y-X+1.
\end{align*}
We already know that $a$ exists for some  $s_0 = s_0(t_0)$.
Then we recover $a$ by
\begin{align}
\log a(s)
=
\log a(s_0)+\int_{s_0}^s (Y-X+1) d\hat{s}.
\end{align}
Because $X < 1$ is preserved the integrand  
is always positive and hence $a$ is positive and uniformly bounded from below. Because $a, X, Y, Z$ are all positive this gives $(a,b,c,f)$. 
Finally we recover the $t$-parameter as
\begin{align*}
t(s)
=
t(s_0)+\int_{s_0}^s
a \sqrt{Y}\ d\hat{s}.
\end{align*}
We know that $a$ is bounded away from zero and the same is true for $Y$ by assumption. Therefore $t\rightarrow\infty$ as $s\rightarrow \infty$. 
We conclude that $(a,b,c,f)$ extends to a forward complete solution of $(\ref{ODE system})$.
\end{proof}

\begin{remark}
\label{critical points}
To conclude this section we discuss the fixed points of the dynamical system (\ref{ODE system in good coord}) and their geometric interpretation. As we consider solutions with positive coefficients 
we only list fixed points with non-negative coordinates.
These are given by
\begin{align}
(0,0,0), (1,0,0), (0,1,0), (1,1,0), \left(\frac{15-3\sqrt{5}}{10},\frac{3-\sqrt{5}}{2},\frac{3\sqrt{5}-5}{5}\right).
\label{fixed points}
\end{align}
Before we move on to describe these in more detail, we quickly review the theory of \textit{hyperbolic} fixed points. For details we refer to \cite[Chapter 2.7]{Perko}. A fixed point $p$ of a dynamical system $\dot{x}=\Phi(x)$ is called hyperbolic if the real parts of all eigenvalues of the linearisation $d\Phi|_p$ of the system at the fixed point are non-zero. If the system is $n$-dimensional and $d\Phi|_p$ has $k$ eigenvalues with negative real part and $(n-k)$-eigenvalues with positive real part, then there is a $k$-dimensional submanifold, the \textit{stable manifold} at $p$, of trajectories converging towards $p$, and a $(n-k)$-dimensional submanifold, the \textit{unstable manifold} at $p$, of trajectories emanating from $p$. Moreover, by the Hartman--Grobman theorem \cite[Chapter 2.8]{Perko} the dynamical system in a neighbourhood of $p$ is topologically equivalent to the linearised system.

All of the fixed points \eqref{fixed points} are hyperbolic:
\begin{itemize}
\item
$(0,1,0)$ has a 1-dimensional stable manifold and a 2-dimensional unstable manifold.
In Section \ref{section-singular orbit} we describe up to scale a 1-parameter family $\Psi_{\mu}$ of smooth cohomogeneity one Spin(7)-structures with principal orbit $N(1,-1)$ closing smoothly on a $S^5$. The trajectories of the associated solutions originate in this fixed point
and sweep out an open subset of the unstable manifold. Therefore, this fixed point can be thought off as the singular orbit $S^5$.
The Bryant--Salamon $\mathrm{G}_2$ holonomy metric on $\Lambda^2_- \CP{2}$, which we have described in Remark \ref{rem-reduced system}, in $(X,Y,Z)$ coordinates also originates in $(0,1,0)$ and is explicitly given by 
\begin{align}
X(s) =  \frac{e^{4s}}{1+e^{4s}},\quad Y(s)=1, \quad Z(s) =0.
\label{BS-XYZ}
\end{align}
One of the two trajectories in the 1-dimensional stable manifold of $(0,1,0)$ is the explicit solution
\begin{align}
X(s) =  0, \quad Y(s)=\frac{e^{4s}}{1+e^{4s}}, \quad Z(s) =0,
\label{Y-axis}
\end{align}
which emanates from $(0,0,0)$.
\item
The dynamics around $(1,0,0)$ are the same as around $(0,1,0)$ and correspond to the singular orbit $\CP{2}$.
A 1-parameter family $\Upsilon_{\tau}$ of smooth cohomogeneity one Spin(7)-structures with principal orbit $N(1,0)$ closing smoothly on a $\CP{2}$ is described in Section \ref{section-singular orbit}. 
The explicit solution
\begin{align}
\label{BS2-XYZ}
X(s) = 1, \quad Y(s)=  \frac{e^{4s}}{1+e^{4s}}, \quad Z(s) =0,
\end{align}
again can be identified with the Bryant--Salamon $\mathrm{G}_2$ holonomy metric on $\Lambda^2_{-}\CP{2}$.
One of the two trajectories in the 1-dimensional stable manifold is the explicit solution
\begin{align}
X(s) = \frac{e^{4s}}{1+e^{4s}}, \quad Y(s) = 0, \quad Z(s) =0,
\label{X-axis}
\end{align}
which emanates from $(0,0,0)$.
\item
$(0,0,0)$ is a source. Once we have understood how to prove Theorems \ref{theorem1} and \ref{thmB}, at the end of section \ref{section-proofs} we show that we can extract 1-parameter families of Spin(7) holonomy metrics emanating from this source which on one end behave as the families $\Psi_{\mu}$ and $\Upsilon_{\tau}$, but are singular on the other end.
The collapsed limit of these families is the the explicit solution 
\begin{align}
\label{diagonal}
X(s) = Y(s) = \frac{e^{4s}}{1+e^{4s}}, \quad Z(s) =0,
\end{align}
which runs along the diagonal in the unit square of the plane $Z=0$. This
is the singular Bryant--Salamon metric, which corresponds
to the case $C=1$ in Remark \ref{rem-reduced system}.
\item 
$(1,1,0)$ is a sink. 
Geometrically this fixed point has two closely related interpretations.
First, the constant solution corresponding to this fixed point is the $\mathrm{G}_2$ holonomy cone $C(F_3)$ over the homogeneous nearly K\"ahler structure on $F_3$, which corresponds to the case $C=0$ in Remark \ref{rem-reduced system}. In accordance with this, solutions \eqref{BS-XYZ}, \eqref{BS2-XYZ} and \ref{diagonal}, which all are AC $\mathrm{G}_2$ holonomy metrics asymptotic to $C(F_3)$, converge to this fixed point as $s\rightarrow\infty$. Secondly, this fixed point can be interpreted as an ALC end, where the asymptotic model is a circle bundle with fibres of a fixed length over the $\mathrm{G}_2$-cone $C(F_3)$. 
In terms of $(a,b,c,f)$ coordinates we have discussed this asymptotic model in Remark \ref{remark-ALC asymp from coeff}.
In terms of $(X,Y,Z)$ coordinates, in Section \ref{section-ALC} we prove 
that all relevant flow lines which converge to $(1,1,0)$ as $s\rightarrow\infty$ indeed are associated with ALC Spin(7)-structures.
\item
The fixed point 
\begin{align*}
(X_c, Y_c, Z_c) := \left(\frac{15-3\sqrt{5}}{10},\frac{3-\sqrt{5}}{2},\frac{3\sqrt{5}-5}{5}\right)
\approx
(0.83, 0.38, 0.34)
\end{align*}
corresponds to the Spin(7)-cone $C$ over the unique $\SU(3)\times \U(1)$-invariant nearly parallel $\mathrm{G}_2$-structure on $N(1,-1)$.
Indeed, the associated solution of the cone solution described in Remark \eqref{remark- AC asymp from coeff} is this fixed point.
The linearisation of the system (\ref{ODE system in good coord})
at $(X_c, Y_c, Z_c)$
is given by 
\begin{align*}
\begin{pmatrix}
-6+\frac{6}{\sqrt{5}} & 0 & -3+\frac{3}{\sqrt{5}}
\\
0 & -6 +\frac{14}{\sqrt{5}} & -5 +\sqrt{5}
\\
3-\frac{9}{\sqrt{5}} & 3-\frac{9}{\sqrt{5}} & 4-\frac{12}{\sqrt{5}}
\end{pmatrix}.
\end{align*}
The eigenvalues rounded to one digit after the decimal point are
$-4.1, -1.7, 1.4$. By the discussion above there is a 2-dimensional stable manifold and a 1-dimensional unstable manifold.
In Section \ref{section-CS} we construct up to scale two cohomogeneity one Spin(7)-metrics with principal orbit $N(1,-1)$ and an isolated conical singularity modelled on the Spin(7)-cone $C$. The two trajectories of the associated solutions constitute the unstable manifold at $(X_c, Y_c, Z_c)$. Furthermore, we show that the 2-dimensional stable manifold is made up of a 2-parameter family $\Psi^{\mathrm{ac}}_{\alpha,\beta}$ of AC ends.
\end{itemize}
\end{remark}

\section{Local solutions around the singular orbits $S^5$ and $\CP{2}$}
\label{section-singular orbit}

In Remark \ref{rem-bundles} we have explained that $N(k,l)$ is a $L(1,|k+l|)$-bundle over $\CP{2}$. For $N(1,-1)$, the fibre $L(1,0)=S^1\times S^2$ is not a sphere. In particular, there is no cohomogeneity one space with principal orbit $N(1,-1)$ and singular orbit $\CP{2}$ (see Section \ref{intro-coh1} for details). 
However, $N(1,-1)$ is an $S^2$-bundle over the 5-sphere $S^5=\SU(3)/\SU(2)$ (see Remark \ref{rem-bundles}). 
Indeed, the adjoint bundle $M_{S^5}$ of the principal $\SU(2)$-bundle $\SU(3)\rightarrow \SU(3)/\SU(2)$ is a cohomogeneity one space with principal orbit $N(1,-1)$ and singular orbit $S^5$.
The group diagram is given by $\U(1)_{1,-1} \subset \SU(2) \subset \SU(3)$. The extra $U(1)$-factor \eqref{extra-U(1)} also is in the normalizer of $\SU(2)$, and therefore gives a global symmetry of $M_{S^5}$.

As explained in Section \ref{intro-coh1}, we want to approach the construction of Spin(7)-metrics on $M_{S^5}$ by first considering local invariant Spin(7)-structures closing smoothly on the singular orbit $S^5$ and then decide which of these extend to complete Spin(7)-structures to all of $M_{S^5}$. Local cohomogeneity one Spin(7)-structures around the singular orbit have been investigated by Reidegeld \cite{Reidegeld-paper}. He proves

\begin{theorem}\cite[Theorem 6.1]{Reidegeld-paper}
\label{reidegeld-short distance existence}
For any $\mu\in(0,\infty)$ there exists a unique
$\SU(3)\times \U(1)$-invariant  torsion-free Spin(7)-structure $\Psi_{\mu}$ in a neighbourhood of the singular orbit $S^5$ in $M_{S^5}$ with
\begin{align*}
a(0) = 0, \quad b(0) = c(0) = 1, \quad f(0) = \mu.
\end{align*} 
The holonomy of the associated metric is all of Spin(7).
$\Psi_{\mu}$ depends continuously on $\mu$.
\end{theorem}

The asymptotic expansion of $\Psi_{\mu}$ is given by
\begin{subequations}
\label{short distance asy exp}
\begin{align}
a(t) =& 2t - \frac{4}{27}(9  - \mu^2) t^3
+ \mathcal{O}(t^5),
\\
b(t) =& 
1  -\frac{1}{3}\mu t
+
\left(
1-\frac{5}{18}\mu^2
\right)
t^2
+
\frac{1}{810}\mu(126 -167 \mu^2)t^3
+
\mathcal{O}(t^4),
\\
c(t) =& 
1  +\frac{1}{3}\mu t
+
\left(
1-\frac{5}{18}\mu^2
\right)
t^2
-
\frac{1}{810}\mu(126 -167 \mu^2)t^3
+\mathcal{O}(t^4),
\\
f(t)
=&
\mu + \frac{2}{3}\mu^3 t^2 + \mathcal{O}(t^4).
\end{align}
\end{subequations}
In the coordinates $(X,Y,Z)$ the short-distance asymptotic expansion takes the form
\begin{subequations}
\label{short distance asy exp in proj coord}
\begin{align}
X(t)
&=
4t^2-\frac{8}{3}\mu t^3+\mathcal{O}(t^4),
\\
Y(t)
&=
1-\frac{4}{3}\mu t+\frac{8}{9}\mu^2 t^2-\frac{8}{405}\mu(83\mu^2-99)t^3+\mathcal{O}(t^4),
\\
Z(t)
&=
2\mu t-\frac{8}{3}\mu^2 t^2+\frac{4}{27}\mu(31\mu^2-36)t^3+\mathcal{O}(t^4).
\end{align}
\end{subequations}

\begin{remark}
\label{mu=0}
$\mu=0$ gives the Bryant--Salamon AC $\mathrm{G}_2$ holonomy metric on $\Lambda^2_{-}\CP{2}$ with $f\equiv 0$ and $b\equiv c$ described in Remark \ref{rem-reduced system}. The continuous dependence of the functions $(a,b,c,f)$ on $\mu$ extends to $\mu=0$.
\end{remark}

\begin{remark}
\label{range of s for S5 sol}
By Remark \ref{associated solutions} each $\Psi_{\mu}$ gives rise to an associated solution of the system (\ref{ODE system in good coord}). By abuse of notation we will denote them by the same symbol $\Psi_{\mu}$. Let us determine the range of parameters $s$ for which these are defined. 
It follows from the asymptotic expansion 
(\ref{short distance asy exp}) that
we can find a positive constant  $C$ and a small time $t_0$ such that for all $t\in(0,t_0)$
\begin{align*}
C t^{-1} < \frac{c}{ab}.
\end{align*}
Set $s(t_0)= s_0$ where $s_0$ is an arbitrary constant of integration. Then
\begin{align*}
s(t) = - \int_{t}^{t_0} \frac{c}{ab} d\hat{t}+s(t_0)
< 
-C \int_t^{t_0} \hat{t}^{-1} d\hat{t} + s(t_0)
=
C \log(t) -C \log(t_0) + s(t_0)
.
\end{align*}
Hence $s\rightarrow -\infty$ as $t\rightarrow 0$.
Therefore, there exists some $S\in\mathbb{R}$ such that 
$\Psi_{\mu}$ is defined for $s\in(-\infty,S)$.
As $s\rightarrow -\infty$, for each $\mu$ the solution $(X, Y, Z)$ converges to the fixed point $(0, 1, 0)$. Hence this fixed point corresponds to a singular orbit $S^5$.
\end{remark}

To use $\CP{2}$ as the singular orbit, we need to use $N(1,0)$ instead of $N(1,-1)$ as the principal orbit. Indeed, by Remark \ref{rem-bundles} $N(1,0)$ is a $L(1,1)=S^3$ bundle over $\CP{2}$, and the universal quotient bundle $M_{\CP{2}}$ is a cohomogeneity one space with principal orbit $N(1,0)$ and singular orbit $\CP{2}$ (see \cite{Gukov-Sparks-Tong}).
The extra $\U(1)$-factor \eqref{extra-U(1)} also is a subgroup of the normalizer of $U(1)_{1,0}$ and of the isotropy group of $\CP{2}$. Therefore, the extra symmetry from Remark \ref{remark-diagonal-metric} acts globally on $M_{\CP{2}}$.

As mentioned in Remark \ref{princ-orbit-N10}, an $\SU(3)$-invariant torsion-free Spin(7)-structure with principal orbit $N(1,0)$ is still characterised as a solution of the system (\ref{ODE system}). We only need to swap the roles of $a$ and $b$ in the discussion of smooth extension over the singular orbit. Taking this into account, Reidegeld proves

\begin{theorem}\cite[Theorem 7.1]{Reidegeld-paper}
\label{reidegeld-short distance existence-CP2}
For any $\tau\in\mathbb{R}$ there exists a unique
$\SU(3)\times \U(1)$-invariant  torsion-free Spin(7)-structure $\Upsilon_{\tau}$  in a neighbourhood of the singular orbit $\CP{2}$
in $M_{\CP{2}}$ with the asymptotic expansion
\begin{subequations}
\label{short-distance-asy-exp-cp2}
\begin{align}
a(t) 
&=
1 + \frac{2}{3} t^2 + \frac{-104-\tau}{288} t^4 + \mathcal{O}(t^5),
\\
b(t)
&=
t-\frac{12+\tau}{24} t^3 +\mathcal{O}(t^5),
\\
c(t)
&=
1+\frac{5}{6}t^2 +\frac{-140+\tau}{288}t^4 + \mathcal{O}(t^5),
\\
f(t)
&=
t+\frac{\tau}{12}t^3+\mathcal{O}(t^5).
\end{align}
\end{subequations}
The holonomy of the associated metric is all of Spin(7).
$\Upsilon_{\tau}$ depends continuously on $\tau$.
\end{theorem}

In $(X,Y,Z)$ coordinates the short-distance expansion takes the form
\begin{subequations}
\label{short-dist-asy-xyz-cp2}
\begin{align}
X(t) &= 1-\frac{1}{3} t^2 - \frac{-40+\tau}{72} t^4 + \mathcal{O}(t^5),
\\
Y(t) &= t^2 - \frac{32+\tau}{12}t^4 +\mathcal{O}(t^5),
\\
Z(t) &= t^2 + \frac{-56+\tau}{24} t^4 + \mathcal{O}(t^5).
\end{align}
\end{subequations}

\begin{remark}
\label{remark-Upsilon-assoc-sol}
Using the asymptotic expansion \eqref{short-distance-asy-exp-cp2}, as in Remark \ref{range of s for S5 sol} we can show that for every $\tau\in\mathbb{R}$ there exists some $S>0$  such that the solution of the system \ref{ODE system in good coord} associated with $\Upsilon_{\tau}$ is defined for $s\in(-\infty,S)$. As $s\rightarrow -\infty$, the solution $(X, Y, Z)$ converges to the fixed point $(1,0,0)$. Hence this fixed point corresponds to the singular orbit $\CP{2}$.
\end{remark}

\begin{remark}
\label{Remark-scaling}
Scaling a Ricci-flat metric by a non-zero positive constant gives another Ricci-flat metric. In the situation of the Spin(7)-structure \eqref{Spin(7)-reduction-G2}, replacing the Spin(7)-form
$\psi=dt\wedge \varphi+ *\varphi$ by $\hat{\psi}=\kappa^4 \psi$ scales the associated metric to $\kappa^2 g= \kappa^2 dt^2 + \kappa^2 h$. Now $\hat{t}=\kappa t$
is the arc-length parameter meeting the principal orbits orthogonally. The scaled Spin(7)-structure $\hat{\psi}$ is represented by the coefficient functions $(\hat{a}(t),\hat{b}(t), \hat{c}(t),\hat{f}(t))=(\kappa\ a(t/\kappa), \kappa\ b(t/\kappa), \kappa\ c(t/\kappa), \kappa\ f(t/\kappa))$.
We are only interested in solutions to (\ref{ODE system}) up to scale. In Theorem \ref{reidegeld-short distance existence} we chose the scale for the family $\Psi_{\mu}$ such that $b(0) = 1$,
and in Theorem \ref{reidegeld-short distance existence-CP2} we chose the scale for the family $\Upsilon_{\tau}$ such that $a(0)=1$.
\end{remark}

\section{Conically singular and asymptotically conical ends}

\label{section-CS}

In this section we will construct families of local $\SU(3)\times \U(1)$-invariant CS and AC Spin(7)-metrics with principal orbit $N(1,-1)$. In both cases the asymptotic cone is the cone over the
unique $\SU(3)\times \U(1)$-invariant nearly parallel $\mathrm{G}_2$-structure on $N(1,-1)$. 
As in \cite{FHN2} this will be achieved by considering a singular initial value problem around the conical singularity in the CS case and at infinity of the asymptotic cone in the AC case. 
The following statement can be found in \cite[Theorem 5.1]{FHN2}.
A proof can be found in Picard's treatise \cite[Chapter I, \S 13]{Picard}.

\begin{theorem}
\label{singular IVP for ends}
Consider the singular initial value problem 
\begin{align}
t \dot{y}=\Phi(y),
\quad
y(0)=y_0,
\label{singular IVP for ends-eq}
\end{align}
where $y$ takes values in $\mathbb{R}^k$ and $\Phi \colon: \mathbb{R}^k \rightarrow \mathbb{R}^k$ is a real analytic function in a neighbourhood of $y_0$
with $\Phi(y_0)=0$. After possibly a change of basis,
assume that $d\Phi|_{y_0}$ contains a diagonal block $\mathrm{diag}(\lambda_1,\dots, \lambda_m)$ in the upper-left corner. Furthermore assume that the eigenvalues $\lambda_1, \dots,\lambda_m$ satisfy:
\begin{compactenum}[(i)]
\item
\label{singular IVP for ends - pos cond}
$\lambda_1, \dots,\lambda_m > 0;$
\item
\label{singular IVP for ends - nonresonance cond}
for every $\boldsymbol{h}=(h_1,\dots, h_m)\in \mathbb{Z}_{\geq 0}^{m}$
with $|\boldsymbol{h}| = h_1 + \dots + h_m \geq 2$ the matrix
\begin{align*}
(\boldsymbol{h}\cdot \boldsymbol{\lambda})\mathrm{Id}-d\Phi|_{y_0}
\end{align*}
is invertible. Here $\boldsymbol{\lambda}=(\lambda_1,\dots, \lambda_m)$ and $\boldsymbol{h} \cdot \boldsymbol{\lambda}=\sum_{i=1}^m h_i \lambda_i$.
\end{compactenum}
Then for every $(u_1, \dots, u_m)\in\mathbb{R}^m$ there exists a unique solution $y(t)$ of (\ref{singular IVP for ends-eq}) given as a convergent generalised power series
\begin{align*}
y(t)=
y_0
+
(u_1 t^{\lambda_1}, \dots, u_m t^{\lambda_m}, 0 \dots 0)
+
\sum_{|\boldsymbol{h}|\geq 2}
y_{\boldsymbol{h}} t^{\boldsymbol{h}\cdot \boldsymbol{\lambda}}.
\end{align*}
Furthermore, the solutions depend real analytically 
on $u_1, \dots , u_m$.
\end{theorem}

In the following, denote by 
$\nu_1, \nu_2, \nu_3$ the ordered roots of the 
cubic equation
\begin{align*}
x^3+8x^2-4x-60=0.
\end{align*}
The numerical values, rounded to two digits after the decimal point, are given by
\begin{align}
\label{num-value-eigenvalues}
\nu_0 \approx -7.46,\quad \nu_1 \approx -3.12,\quad \nu_2 \approx 2.58.
\end{align}

\begin{proposition}
\label{CS and AC ends}
Let $C$ be the Spin(7)-holonomy cone over $N(1,-1)$.
\begin{compactenum}[(i)]
\item 
For every $\lambda\in \mathbb{R}$ there is some $\varepsilon > 0$ such that on $(0,\varepsilon)\times N(1,-1)$ there exists a torsion-free CS Spin(7)-structure $\Psi^{\mathrm{cs}}_{\lambda}$
asymptotic to $C$ 
which has the asymptotic expansion
\begin{subequations}
\label{short time asymp for CS sol}
\begin{align}
\frac{\sqrt{5}}{2} t^{-1}a(t)
&
\approx 
1- 0.25 \lambda\ t^{\nu_2}+\mathcal{O}(t^{2\nu_2}),
\\
\sqrt{\frac{15}{2(5-\sqrt{5})}}t^{-1}b(t) 
&
\approx 
1-4.84 \lambda\ t^{\nu_2}+\mathcal{O}(t^{2\nu_2}),
\\
\sqrt{\frac{15}{2(5+\sqrt{5})}}t^{-1}c(t) 
&
\approx
1+0.09 \lambda\ t^{\nu_2}+\mathcal{O}(t^{2\nu_2}),
\\
\frac{3\sqrt{5}}{4}t^{-1}f(t) 
&
\approx 
1+10 \lambda\ t^{\nu_2}+\mathcal{O}(t^{2\nu_2}).
\end{align}
\end{subequations}
Here all coefficients have been rounded to two digits after the decimal point.
\item
For every $(\alpha,\beta)\in\mathbb{R}^2$ there is some $T > 0$ such that on $(T,\infty)\times N(1,-1)$ there exists a torsion-free AC Spin(7)-structure $\Psi^{\mathrm{ac}}_{\alpha,\beta}$ asymptotic to $C$ which has the asymptotic expansion
\begin{subequations}
\label{asy-exp-AC-ends}
\begin{align}
\frac{\sqrt{5}}{2} t^{-1}a(t)
&
\approx 
1- 10.6\alpha\ t^{\nu_1}+3.6 \beta\ t^{\nu_0} +\sum_{k,l\geq 0, k+l\geq 2} a_{kl}\, t^{k\nu_1 + l \nu_0},
\\
\sqrt{\frac{15}{2(5-\sqrt{5})}}t^{-1}b(t) 
&
\approx 
1+10.8\alpha\ t^{\nu_1}+0.8\beta\ t^{\nu_0}
+\sum_{k,l\geq 0, k+l\geq 2} b_{kl}\, t^{k\nu_1 + l \nu_0}
,
\\
\sqrt{\frac{15}{2(5+\sqrt{5})}}t^{-1}c(t) 
&
\approx
1-5.1\alpha\ t^{\nu_1}-4.8\beta\ t^{\nu_0}
+\sum_{k,l\geq 0, k+l\geq 2} c_{kl}\, t^{k\nu_1 + l \nu_0}
,
\\
\frac{3\sqrt{5}}{4}t^{-1}f(t) 
&
\approx
1+10\alpha\ t^{\nu_1}+ \beta\ t^{\nu_0}
+\sum_{k,l\geq 0, k+l\geq 2} f_{kl}\, t^{k\nu_1 + l \nu_0}
.
\end{align}
\end{subequations}
Here the leading coefficients have been rounded to one digit after the decimal point and the higher coefficients $a_{kl}, b_{kl}, c_{kl}, f_{kl}$ are determined by $(\alpha,\beta)$.
If $\alpha = 0$, $\Psi^{\mathrm{ac}}_{\alpha,\beta}$ has decay rate $\nu_0$, otherwise it has decay rate $\nu_1$.
\end{compactenum}
\end{proposition}
\begin{proof}
Recall from Remark \ref{remark- AC asymp from coeff}
that the cone over the $\SU(3)\times \U(1)$-invariant nearly parallel $G_2$-structure on $N(1,-1)$ is given by
\begin{align*}
a= \frac{2}{\sqrt{5}}\, t,
\quad
b = \sqrt{\frac{2}{15}(5-\sqrt{5})}\, t,
\quad
c = \sqrt{\frac{2}{15}(5+\sqrt{5})}\, t,
\quad
f = \frac{4}{3\sqrt{5}}\, t.
\end{align*}
Therefore, any deformation of the conical Spin(7)-structure on $(0,\infty)\times N(1,-1)$ can be described as
\begin{align*}
t^{-1}a= \frac{2}{\sqrt{5}}(1+X_1),
\quad
t^{-1}b = \sqrt{\frac{2}{15}(5-\sqrt{5})}(1+X_2),
\\
t^{-1}c = \sqrt{\frac{2}{15}(5+\sqrt{5})}(1+X_3),
\quad
t^{-1}f = \frac{8}{3\sqrt{5}}(1+X_4).
\end{align*}
Setting $(X_1, X_2, X_3, X_4) = (0,0,0,0)$ recovers the Spin(7)-cone. The system (\ref{ODE system}) becomes
\begin{align*}
t \dot{X_1}
=&
-X_1
+
\frac{5-\sqrt{5}}{4}\frac{1+X_2}{1+X_3}
+
\frac{5+\sqrt{5}}{4}\frac{1+X_3}{1+X_2}
-
\frac{3}{2}
\frac{1+X_1}{1+X_2}\frac{1+X_1}{1+X_3}-1,
\\
t \dot{X_2}
=&
-X_2
+
\frac{5+\sqrt{5}}{4}\frac{1+X_3}{1+X_1}
+
\frac{3}{2}\frac{1+X_1}{1+X_3}
-
\frac{5-\sqrt{5}}{4}
\frac{1+X_2}{1+X_3}\frac{1+X_2}{1+X_1}
-
\frac{2}{\sqrt{5}-1}\frac{1+X_4}{1+X_2}
-1,
\\
t \dot{X_3}
=&
-X_3
+
\frac{3}{2}\frac{1+X_1}{1+X_2}
+
\frac{5-\sqrt{5}}{4}
\frac{1+X_2}{1+X_1}
-
\frac{5+\sqrt{5}}{4}\frac{1+X_3}{1+X_1}
\frac{1+X_3}{1+X_2}
+
\frac{2}{\sqrt{5}+1}\frac{1+X_4}{1+X_3}
-1,
\\
t \dot{X}_4
=&
-X_4
+
\frac{2}{\sqrt{5}-1}\frac{(1+X_4)^2}{(1+X_2)^2}
-
\frac{2}{\sqrt{5}+1}\frac{(1+X_4)^2}{(1+X_3)^2}
-1.
\end{align*}
 
The linearisation $L$ of the right-hand side at $(0,0,0,0)$ is given by
\begin{align*}
L=
\begin{pmatrix}
-4 & \frac{-\sqrt{5}+3}{2}
& \frac{\sqrt{5}+3}{2} & 0
\\
\frac{-\sqrt{5}+3}{2} & \sqrt{5}-3
& 1 &  -\frac{2}{\sqrt{5}-1}
\\
\frac{\sqrt{5}+3}{2} & 1
& -\sqrt{5}-3
& \frac{2}{\sqrt{5}+1}
\\
0 & -\sqrt{5}-1
& \sqrt{5}-1  & 1
\end{pmatrix}.
\end{align*}
The eigenvalues of $L$ are given by $\nu_0, \nu_1, \nu_2$ and $-1$.
Writing $y = (X_1, X_2, X_3, X_4)$,
this is a system of the form (\ref{singular IVP for ends-eq}) with $y_0 = (0,0,0,0)$.

We will first construct the family of CS solutions.
The numerical values (\ref{num-value-eigenvalues}) show that condition (ii) in Theorem \ref{singular IVP for ends} is satisfied 
if we set $m=1$ and $\lambda_1 = \nu_2$. The eigenspace of $L$ associated with $\nu_2$ is spanned by $(-0.25,-4.84,0.09,10)$, where all components are rounded to two digits after the decimal point.
The existence of the 1-parameter family $\Psi^{\mathrm{cs}}_{\lambda}$ follows from Theorem \ref{singular IVP for ends}.

We have to replace $t$ by $1/t$ to construct the AC ends with Theorem \ref{singular IVP for ends}. Then the linearisation is given by $-L$.
By using the numerical approximations (\ref{num-value-eigenvalues}), one can see that the non-resonance condition (\ref{singular IVP for ends - nonresonance cond}) of Theorem \ref{singular IVP for ends} is satisfied
if we set $m=2$ and $\lambda_1= -\nu_0$, $\lambda_2 = -\nu_1$. Rounded to one digit after the decimal point, the eigenspaces associated with $\nu_0$ and $\nu_1$ are spanned by the vectors $(3.6, 0.8, -4.8, 1)$ and $(-10.6, 10.8, -5.1, 10)$, respectively. 
The statement follows with Theorem \ref{singular IVP for ends}.
\end{proof}

The solution $\Psi^{\mathrm{cs}}_0$ is the Spin(7)-cone itself.
All solutions $\Psi^{\mathrm{cs}}_{\lambda}$, $\lambda > 0$, are related by scaling, as are all solutions $\Psi^{\mathrm{cs}}_{\lambda}$, $\lambda < 0$. 
By Remark \ref{associated solutions} each $\Psi^{\mathrm{cs}}_{\lambda}$ gives rise to an associated solution of the system (\ref{ODE system in good coord}). Because by passing to $(X,Y,Z)$ coordinates Spin(7)-structures related by scaling are identified, we only get three distinct solutions and different choices of $\lambda$ of the same sign merely correspond to a shift in the $s$-parameter. The associated solution of
the Spin(7)-cone $\Psi^{\mathrm{cs}}_0$ is the fixed point $(X_c, Y_c, Z_c)$, which we have described in Remark \ref{critical points}. To determine the remaining two trajectories corresponding to the conically singular solutions $\Psi^{\mathrm{cs}}_{\lambda}$, we use the asymptotic expansion (\ref{short time asymp for CS sol}) to argue as in Remark \ref{range of s for S5 sol} that for each $\Psi^{\mathrm{cs}}_{\lambda}$ we can find an $S\in\mathbb{R}$ such that the associated solution $(X,Y,Z)$ is defined for $s\in(-\infty,S)$. As $s\rightarrow -\infty$, the solution $(X,Y,Z)$ converges to the fixed point $(X_c, Y_c, Z_c)$, which corresponds to the Spin(7)-cone. Thus, the two trajectories associated with the families  $\Psi^{\mathrm{cs}}_{\lambda}$, $\lambda > 0$, and $\Psi^{\mathrm{cs}}_{\lambda}$, $\lambda < 0$, are precisely the two branches of the 1-dimensional unstable manifold at $(X_c, Y_c, Z_c)$.

It will still be useful to us to compute the asymptotic expansion with respect to the $t$-parameter  of the associated solution of $\Psi^{\mathrm{cs}}_{\lambda}$ 
 as $t\rightarrow 0$, which is given by
\begin{subequations}
\label{short time asy exp cs sol}
\begin{align}
X(t)
&
\approx
X_c(1-0.68 \lambda t^{\nu_2}+\mathcal{O}(t^{2\nu_2})),
\\
Y(t)
&
\approx
Y_c(1-9.86 \lambda t^{\nu_2}+\mathcal{O}(t^{2\nu_2})),
\\
Z(t) 
&
\approx
Z_c (1 + 4.46 \lambda t^{\nu_2} + \mathcal{O}(t^{2\nu_2})).
\end{align}
\end{subequations}
Again all coefficients are rounded to two digits after the decimal point.

The solution $\Psi^{\mathrm{ac}}_{0,0}$ is the Spin(7)-cone itself.
Next we determine which ones of the AC ends $\Psi^{\mathrm{ac}}_{\alpha,\beta}$, $(\alpha,\beta) \neq (0,0)$, are related by scaling.
In Remark \ref{Remark-scaling} we have described how the Spin(7)-structures scale. After rescaling by $\kappa > 0$, the AC end $\kappa\Psi^{\mathrm{ac}}_{\alpha,\beta}$
is described by the functions
\begin{align*}
(\hat{a}(t),\hat{b}(t), \hat{c}(t),\hat{f}(t))=(\kappa\ a(t/\kappa), \kappa\ b(t/\kappa), \kappa\ c(t/\kappa), \kappa\ f(t/\kappa)).
\end{align*}
This corresponds to replacing the functions $(X_1(t), X_2(t), X_3(t),X_4(t))$ from the proof of Proposition \ref{CS and AC ends}
by  
$(X_1(t/\kappa),X_2(t/\kappa), X_3(t/\kappa), X_4(t/\kappa))$. 
By the asymptotic expansion \eqref{asy-exp-AC-ends}, the latter quadruple can also be obtained by replacing the parameters $(\alpha,\beta)$ by $(\kappa^{-\nu_1} \alpha,\kappa^{-\nu_0} \beta)$. Hence the orbits of the action of $\mathbb{R}_{+}$ on $\mathbb{R}^2-\{(0,0)\}$ given by $\kappa \circ (\alpha,\beta)=(\kappa^{-\nu_1} \alpha, \kappa^{-\nu_0} \beta)$ consist  precisely of parameters corresponding to AC ends which are related by scaling.  
The quotient of $\mathbb{R}^2-\{(0,0)\}$ by this action is homeomorphic to $S^1$ and each orbit has a unique representative on $S^1 \subset \mathbb{R}^2$. 

Using the asymptotic expansion \eqref{asy-exp-AC-ends}, again we can argue similarly as in Remark \ref{range of s for S5 sol} to deduce that for each $(\alpha,\beta)$ there exists $S\in\mathbb{R}$ such that the solution  $(X_{\alpha,\beta}(s), Y_{\alpha,\beta}(s), Z_{\alpha,\beta}(s))$ associated with $\Psi^{\mathrm{ac}}_{\alpha,\beta}$ this time is defined for $s\in(S,\infty)$. As $s\rightarrow \infty$, the solution approaches the Spin(7)-cone $(X_c, Y_c, Z_c)$.
Furthermore, because $S^1$ is compact, we can find $s_0\in\mathbb{R}$ independent of $(\alpha,\beta)\in S^1 \subset \mathbb{R}^2$ such that   
the associated solution of $\Psi^{\mathrm{ac}}_{\alpha,\beta}$ for any $(\alpha,\beta)\in S^1$ is defined for  $s\in (s_0,\infty)$.
The map 
\begin{gather*}
S^1 \rightarrow \mathbb{R}^3,
\\
(\alpha,\beta) \mapsto (X_{\alpha,\beta}(s_0), Y_{\alpha,\beta}(s_0), Z_{\alpha,\beta}(s_0)),
\end{gather*}
is an embedding. As we increase the choice of $s_0$, this embedded circle sweeps out a punctured embedded 2-ball centred at  $(X_c, Y_c, Z_c)$. Thus, the 2-dimensional stable manifold at the fixed point $(X_c, Y_c, Z_c)$ corresponds precisely to the trajectories associated with the AC ends   $\Psi^{\mathrm{ac}}_{\alpha,\beta}$.
This proves

\begin{proposition}
\label{lemma - crit point cone is AC}
Let $(X,Y,Z)$ be a forward complete solution of the system (\ref{ODE system in good coord}) which converges to the fixed point $(X_c, Y_c, Z_c)$. Then there exist $\alpha_0, \beta_0\in \mathbb{R}$ such that $(X,Y,Z)$ is associated to the AC end $\Psi^{\mathrm{ac}}_{\alpha_0,\beta_0}$.
\end{proposition}

In the construction of the AC ends in the proof of Proposition \ref{CS and AC ends}, the linearisation $-L$ also has the positive eigenvalue $1$.
Deformations given by this eigenvalue correspond to translations of the $t$-variable, and therefore do not give new solutions. In particular we get

\begin{corollary}
\label{class-decay-rate}
There exists a gauge such that any $\SU(3)\times \U(1)$-invariant AC Spin(7) metric which is asymptotic to the cone over the unique $\SU(3)\times \U(1)$-invariant nearly parallel $\mathrm{G}_2$-structure on $N(1,-1)$ has decay rate equal to $\nu_0$ or $\nu_1$.
\end{corollary}

\section{ALC asymptotics}
\label{section-ALC}

In this section we are going to show that any forward complete solution $(a,b,c,f)$ of the ODE system (\ref{ODE system}) describes an ALC Spin(7)-structure if and only if the associated solution $(X,Y,Z)$ of the system (\ref{ODE system in good coord}) converges to the fixed point $(1,1,0)$ as $s\rightarrow \infty$.

\begin{lemma}
Let $(a,b,c,f)$ be a solution of the system (\ref{ODE system}), where $a,b,c,f$ are positive functions satisfying $a, b < c$.
If the associated solution $(X,Y,Z)$ of the system (\ref{ODE system in good coord}) is forward complete with
$\lim_{s\rightarrow \infty} (X,Y,Z) = (1,1,0)$, then $(a,b,c,f)$ is forward complete and there exists $\ell >0$ such that
\begin{align}
\lim_{t\rightarrow\infty} \frac{a(t)}{t} = 1,
\quad
\lim_{t\rightarrow\infty} \frac{b(t)}{t} = 1,
\quad
\lim_{t\rightarrow\infty} \frac{c(t)}{t} = 1,
\quad
\lim_{t\rightarrow\infty} f(t) = \ell.
\label{limit-abcf}
\end{align}
\end{lemma}
\begin{proof}
$(a,b,c,f)$ is forward complete by Lemma \ref{lemma-reduction to proj system}.
The relation 
\begin{align*}
(X,Y,Z)=(a^2/c^2,b^2/c^2, abf/c^3)
\end{align*}
allows us to substitute $\lim_{s\rightarrow \infty} (X,Y,Z) = (1,1,0)$ into the right-hand side of the ODE system (\ref{ODE system}) to obtain 
\begin{align*}
\lim_{t\rightarrow\infty} \dot{a}(t) = 1,
\quad
\lim_{t\rightarrow\infty} \dot{b}(t) = 1,
\quad
\lim_{t\rightarrow\infty} \dot{c}(t) = 1.
\end{align*}
This proves the assertion for $a,b$ and $c$.

Next we show that $f$ is bounded. 
By assumption
\begin{align*}
Z = \frac{abf}{c^3} =
\left(
\frac{a}{t}
\right)
\left(
\frac{b}{t}
\right)
\left(
\frac{t}{c}
\right)^3
\frac{f}{t}
\end{align*}
converges to zero. By the already established limiting behaviour on $a,b,c$, we see that $f$ grows at most as $o(t)$. In particular,
there exists some $\kappa > 0$ such that for sufficiently large times 
$f(t) < t/\kappa$. 
To see that $f$ is bounded we write 
\begin{align*}
\dot{f}=\frac{f^2}{t^2}
\frac{t^2}{b^2}
(1-Y).
\end{align*}
Because $\lim_{t\rightarrow\infty} Y(t) = 1$ and $\lim_{t\rightarrow\infty} b(t)/t=1$, we see that for sufficiently large times we have $f(t) < t/\kappa$ and 
\begin{align*}
\dot{f} < \kappa \frac{f^2}{t^2}.
\end{align*}
The boundedness of $f$ follows from Lemma \ref{comparison lemma}.
Because we have $b<c$, $f$ is monotone increasing by Lemma \ref{first observations}. Since $f$ is bounded, monotone increasing and positive, it will converge to some positive constant $\ell$.
\end{proof}

We have used the following comparison principle.
\begin{lemma}
\label{comparison lemma}
Let $f\in \mathcal{C}^1([T,\infty))$ with $T>0$. If
there exist  $t_0 > T$ and $C>0$ such that 
$\dot{f} < C f^2/t^2$   
for all $t\in [t_0,\infty)$ and $f(t_0) < t_0 / C$, 
then $f$ is bounded from above.
\end{lemma}
\begin{proof}
The solution of the model equation 
$\dot{h} = C h^2/t^2$ is
\begin{align*}
h_{\alpha}(t) = \frac{t}{C - \alpha t}
\end{align*}
where $\alpha$ is the constant of integration.
For us only $\alpha < 0$ is relevant. In this case $h_{\alpha}$ has a singularity at $t^*=C/\alpha < 0$
with $\lim_{t \nearrow t^*} h(t) = \infty$ and 
$\lim_{t \searrow t^*} h(t) = - \infty$ and is asymptotic to
$1/|\alpha|$ for $t \rightarrow \pm \infty$.

If $f(t_0) < t_0 / C$
for some $t_0 > T$,
we can find an $\alpha < 0$ to make this inequality slightly stronger: 
\begin{align*}
f(t_0) < \frac{t_0}{C-\alpha t_0} = h_{\alpha}(t_0).
\end{align*}
By the above discussion $h_{\alpha}$ is smooth for $t\geq 0$
because $\alpha < 0$ and bounded from above by $1/|\alpha|$.
Hence for all $t \geq t_0$ we have the bound $f(t) < h_{\alpha}(t) < 1/|\alpha|$. It follows that $f$ is bounded.
\end{proof}

\begin{proposition}
Assume that $(a, b, c, f)$ is a forward complete solution of the system \eqref{ODE system} which satisfies \eqref{limit-abcf}.
Write
\begin{align*}
\tilde{a}(t) = t^{-1} a(t) - 1,
\quad
\tilde{b}(t) = t^{-1} b(t) - 1,
\quad
\tilde{c}(t) = t^{-1} c(t) - 1,
\quad
\tilde{f}(t) = \frac{1}{\ell} f(t) - 1.
\end{align*}
Then there exists $\gamma > 0$ such that $\tilde{a}^{(k)}(t), \tilde{b}^{(k)}(t),\tilde{c}^{(k)}(t),
\tilde{f}^{(k)}(t)$ behave like $\mathcal{O}(t^{-k-\gamma})$ as $t\rightarrow\infty$ for $k \geq 0$.
Here $\tilde{a}^{(k)}(t)$ denotes the k-th derivative of $\tilde{a}(t)$.
\end{proposition}
\begin{proof}
Set
\begin{align*}
a(t) = t (1+X_1(t)),
\quad
b(t) = t (1+X_2(t)),
\quad
c(t) = t (1+X_3(t)),
\quad
f(t) = \ell (1+X_4(t)).
\end{align*}
The assumption \eqref{limit-abcf} is equivalent to
\begin{align*}
\lim_{t\rightarrow\infty}X_i(t)=0,\quad \text{for}\ i=1,\dots,4
\end{align*}
After the change of variable $e^{\tau} = t$ the system (\ref{ODE system}) becomes
\begin{align*}
\frac{dX_1}{d\tau}
&=
-X_1+\frac{(1+X_2)^2+(1+X_3)^2-(1+X_1)^2}{(1+X_2)(1+X_3)}-1,
\\
\frac{d X_2}{d\tau}
&=
-X_2+\frac{(1+X_3)^2+(1+X_1)^2-(1+X_2)^2}{(1+X_3)(1+X_1)}
-\ell\, e^{-\tau} \frac{1+X_4}{1+X_2}
-1,
\\
\frac{d X_3}{d \tau}
&=
-X_3+\frac{(1+X_1)^2+(1+X_2)^2-(1+X_3)^2}{(1+X_1)(1+X_2)}
+\ell\, e^{-\tau} \frac{1+X_4}{1+X_3}
-1,
\\
\frac{d X_4}{d\tau}
&=
\ell\, e^{-\tau}\frac{(1+X_4)^2}{(1+X_2)^2}
-
\ell\, e^{-\tau}\frac{(1+X_4)^2}{(1+X_3)^2}.
\end{align*}
Setting $X_5=e^{-\tau}$ and $X=(X_1,X_2,X_3,X_4,X_5)$, we get a system of equations of the form $\frac{dX}{d\tau}= \Phi(X)$, where $X(0)=0$ and the linearisation of $\Phi$ at 0 is given by
\begin{align*}
d\Phi|_{U=0}
=
\begin{pmatrix}
-3 & 1 & 1 & 0 & 0
\\
1 & -3 & 1 & 0 & -\ell
\\
1 & 1 & -3 & 0 & \ell
\\
0 & 0 & 0 & 0 & 0
\\
0 & 0 & 0 & 0 & -1
\end{pmatrix}.
\end{align*}
$d\Phi_{U=0}$ has a 1-dimensional kernel spanned by $(0,0,0,1,0)$
and four negative eigenvalues.
Moreover, $\{(0,0,0,c,0)\,|\, c\in\mathbb{R}\}$ is the center manifold of the system. The center manifold equation is
\begin{align*}
\frac{d X_4}{d\tau}=0.
\end{align*}
Hence by \cite[Theorem 2]{Carr} for any solution $X$ converging to the stationary point $X=0$ as in our hypothesis there exists $\gamma>0$ such that
\begin{align*}
(X_1,X_2,X_3,X_4,X_5) = (0,0,0,0,0) + \mathcal{O}(e^{-\gamma\tau}).
\end{align*}
The polynomial decay follows by switching back to the variable $t$. 
The argument for the derivatives of $(X_1, X_2, X_3, X_4)$ follows from a bootstrap argument.
\end{proof}

The results in this section prove

\begin{proposition}
\label{prop-ALC crit point}
Let $(a,b,c,f)$ be a solution of the system (\ref{ODE system}), where $a,b,c,f$ are positive functions satisfying $a, b < c$.
Suppose the associated solution $(X,Y,Z)$ of the system (\ref{ODE system in good coord}) is forward complete with
\begin{align*}
\lim_{s\rightarrow \infty} (X,Y,Z) = (1,1,0).
\end{align*} 
Then $(a, b, c, f)$ defines an $\SU(3)\times \U(1)$-invariant ALC Spin(7) metric on $(T,\infty)\times N(1,-1)$ for some $T>0$.
\end{proposition}

\section{Analysis of the ODE system}

\label{section-alaysis}

In the remainder of the paper we want to investigate which members of the families $\Psi_{\mu}$, $\Upsilon_{\tau}$ and $\Psi^{\mathrm{cs}}_{\lambda}$
give rise to forward complete Spin(7)-holonomy metrics and determine the asymptotic type of complete solutions. 
Once we have proven Theorems \ref{theorem1} and \ref{thmB}, it will be clear how to construct the families in Theorem \ref{thm-sing}.
As discussed in Section \ref{section-coord change}, solutions of the ODE system (\ref{ODE system}) are best studied by looking at their associated solutions of the system (\ref{ODE system in good coord}) described in Remarks \ref{range of s for S5 sol}, \ref{remark-Upsilon-assoc-sol} and at the end of section \ref{section-CS}.
 
The following Lemma will allow us to compare 
the local solutions for different parameters.

\begin{lemma}
\label{comparison of solutions}
Suppose $(X_1,Y_1,Z_1)$ and $(X_2,Y_2,Z_2)$ are two solutions of the system (\ref{ODE system in good coord}), where all functions are positive. Furthermore, suppose that at some 
time $s_0\in\mathbb{R}$ we have
\begin{align*}
X_1 > X_2,
\quad
Y_1 > Y_2,
\quad
Z_1 < Z_2.
\end{align*}
Then this condition is forward preserved as long as all functions stay positive.
\end{lemma}
\begin{proof}
We start by looking at the quantity $Z$.
Assume all three inequalities are preserved until some time $s_1 > s_0$ 
when we have $Z_1(s_1)=Z_2(s_1)=\alpha>0$.
Note that at $s_1$ there must be a strict inequality for either $X_1 > X_2$ or $Y_1 > Y_2$. Otherwise the solutions would be the same. Furthermore at $s_1$ we have
\begin{align*}
\dot{Z}_2-\dot{Z}_1
=
3\alpha
((X_1-X_2)+(Y_1-Y_2)) > 0.
\end{align*}
Therefore, $Z_1 < Z_2$ is strictly preserved as long as $X_1 \geq X_2$ and $Y_1 \geq Y_2$. Given that, suppose that at $s_1 > s_0$ we have  $X_1(s_1)=X_2(s_1)=\alpha>0$.
Then at the same time
\begin{align*}
\dot{X}_2-\dot{X}_1
=
2\alpha(Z_1-Z_2) < 0.
\end{align*}
If we suppose that at $s_1 > s_0$ we have  $Y_1(s_1)=Y_2(s_1)=\alpha>0$, then at the same time
\begin{align*}
\dot{Y}_2-\dot{Y}_1
=
-2(1+\alpha)(Z_2-Z_1) < 0.
\end{align*}
All cases lead to a contradiction.
\end{proof}

As an immediate application of Lemma \ref{comparison of solutions} we obtain

\begin{lemma}
\label{comparison for Phu-mu}
Denote by $(X_1,Y_1,Z_1)$ and $(X_2,Y_2,Z_2)$ the two solutions of the system (\ref{ODE system in good coord}) corresponding to 
\begin{itemize}
\item 
$\Psi_{\mu_1}$ and $\Psi_{\mu_2}$ for $0 < \mu_1 < \mu_2$, respectively, or
\item
$\Upsilon_{\tau_1}$ and $\Upsilon_{\tau_2}$ for $\tau_1 < \tau_2$, respectively.
\end{itemize}
In both cases we have  
\begin{align*}
X_1 > X_2,
\quad
Y_1 > Y_2,
\quad
Z_1 < Z_2,
\end{align*}
as long as the solutions exist.
\end{lemma}
\begin{proof}
By the short distance asymptotic expansions (\ref{short distance asy exp in proj coord}), (\ref{short-dist-asy-xyz-cp2}) the statement is true for small times. 
By Lemma \ref{comparison of solutions} this is preserved as long as all functions stay positive. Any of the $(X,Y,Z)$ coordinates becoming zero means that one of the corresponding functions $a, b, c, f$ must be zero, and thus that the respective solution develops a singularity.
\end{proof}

A simple consequence of the above comparison argument is that the families $\Psi_{\mu}$ and $\Upsilon_{\tau}$ contain at most one AC space.

\begin{lemma}
\label{uniqueness-simple}
Suppose $(X_1,Y_1,Z_1)$ and $(X_2,Y_2,Z_2)$ are two complete solutions of the system (\ref{ODE system in good coord}), where all functions are positive and satisfy 
\begin{align*}
X_1 > X_2,
\quad
Y_1 > Y_2,
\quad
Z_1 < Z_2,
\end{align*}
for all times. Then not both solutions can converge to the cone $(X_c, Y_c, Z_c)$ as $s\rightarrow \infty$.
\end{lemma}
\begin{proof}
The positivity and the given ordering of the two solutions imply
\begin{align*}
\dot{Y}_1-\dot{Y}_2
&=
4(Y_1-Y_2)-4(Y_1^2-Y_2^2)-2 Z_1 (1+Y_1)+2 Z_2 (1+Y_2)
\\
&>
4(Y_1-Y_2)-4(Y_1^2-Y_2^2)-2 Z_1 (1+Y_1)+2 Z_1 (1+Y_2)
\\
&=
(Y_1-Y_2)(4-4Y_1-4Y_2-2 Z_1).
\end{align*}
If both solutions converge to the cone,
for $\epsilon>0$ arbitrarily small and all sufficiently large times we have
\begin{align*}
4-4Y_1-4 Y_2 - 2 Z_1 > 4 - 8 Y_c -2 Z_c - \varepsilon
=\frac{14}{\sqrt{5}}-6-\varepsilon \approx 0.26 -\varepsilon.
\end{align*}
Therefore, the function $Y_1-Y_2$ is monotone increasing for large times. In particular, $Y_1$ and $Y_2$ cannot have the same limit which is a contradiction.
\end{proof}

A further fundamental property of the dynamical system \eqref{ODE system in good coord} is that complete trajectories of interest are flow lines connecting fixed points.

\begin{proposition}\label{prop-crit point convergence}
Let $(X,Y,Z)$ be a solution of the system (\ref{ODE system in good coord}) which is contained inside a compact subset of $\mathbb{R}^3$ for all 
times and in particular complete.
Then the solution converges to a fixed point as $s\rightarrow\infty$.
\end{proposition}
\begin{proof}
We will first show that all three functions $X(s), Y(s), Z(s)$ eventually become monotone.
At a critical point of $X(s),Y(s)$ or $Z(s)$, the respective second derivative is given by
\begin{align}
\frac{1}{2}\frac{d^2}{ds^2}\log X
&= 
-\dot{Z},
\\
\frac{d^2}{ds^2} Y 
&= -2(1+Y) \dot{Z},
\\
\frac{d^2}{ds^2}\log Z
&=
-3\dot{X}-3\dot{Y}.
\end{align}
If for example we denote by $(+-+)$ the chamber where $X(s)$ is increasing, $Y(s)$ decreasing and $Z(s)$ increasing, we get the following diagram:

\begin{center}
\begin{tikzcd}[row sep=scriptsize, column sep=scriptsize]
(-++) \arrow[rr,leftarrow] \arrow[dd,leftrightarrow] \arrow[dr,rightarrow]  && 
(+++) \arrow[dd] \arrow[dr]
& 
\\
& 
(--+)  \arrow[rr,crossing over, leftarrow] 
& & 
(+-+) \arrow[dd,leftrightarrow]
\\
(-+-) \arrow[rr] \arrow[dr,leftarrow] 
& & 
(++-) \arrow[dr,leftarrow]
&
\\
& (---) \arrow[uu,crossing over] \arrow[rr] & & (+--)
\end{tikzcd}
\end{center}
Here an arrow between two chambers indicates that a solution of (\ref{ODE system in good coord}) can transition from one chamber to the other in the direction of the arrow.
The only cycles in the diagram are 
$(-+-) \rightarrow (-++) \rightarrow (-+-)$ and
$(+--)\rightarrow(+-+)\rightarrow(+--)$. We see that $X$ and $Y$ eventually become monotone and only $Z$ can possibly oscillate.
Because $X$ and $Y$ are bounded and monotone, there exist $X_{\infty}$ and 
$Y_{\infty}$ such that $X\rightarrow X_{\infty}$ and $Y\rightarrow Y_{\infty}$ as $s\rightarrow \infty$.

If we set $L=\frac{1}{4}(5-3X_{\infty}-3Y_{\infty})$, then after a finite time
\begin{align*}
4Z(L-\varepsilon -Z) < \dot{Z} < 4Z (L+\varepsilon-Z)
\end{align*}
for arbitrarily small $\varepsilon > 0$.
If at a sufficiently large time $Z < L-\varepsilon$, then either $Z$ is monotone increasing from then on or $Z > L-\varepsilon$ after some time, which is preserved. If at a sufficiently large time $Z > L+\varepsilon$, then either $Z$ is monotone decreasing from then on or $Z < L-\varepsilon$ after some time, which is preserved. We can conclude that either $Z$ becomes monotone or converges to $L$. Because $Z$ is globally bounded, in either case $Z$ converges to some $Z_{\infty}$ as $s\rightarrow\infty$. It is clear that $(X_{\infty},Y_{\infty},Z_{\infty})$ is a fixed point of the system (\ref{ODE system in good coord}).
\end{proof}

By the asymptotic expansions (\ref{short distance asy exp in proj coord}) and \eqref{short-dist-asy-xyz-cp2} for $\Psi_{\mu}$ and $\Upsilon_{\tau}$, respectively, the trajectories of the associated solutions of the system \eqref{ODE system in good coord} initially start out in the cube $0 < X < 1, 0 < Y < 1, 0 < Z < 5/4$.
By Lemma \ref{bounds in xyz coords} these trajectories either exit this cube at the face $Y=0$, where the associated Spin(7)-structure degenerates, or the trajectory is complete and by Proposition \ref{prop-crit point convergence}
converges to a fixed point as $s\rightarrow \infty$. Therefore, the growth behaviour of the function $Y(s)$ of the solution of interest is of particular importance. If for instance $Y(s)$ becomes monotone increasing
for large $s$ while the solution is still contained inside the cube, the  solution must be complete. This motivates us to study extrema of the function $Y(s)$.

\begin{lemma}
\label{lemma-growth of Y}
\begin{compactenum}[(i)]
\item
If $Y \geq 0$, for the growth of $Y$ we have 
\begin{align*}
&\dot{Y} < (>,=)\ 0
\quad
\Leftrightarrow
\quad
Z >(<,=)\ 2Y\frac{1-Y}{1+Y}.
\end{align*}
\item In the following we set
\begin{align*}
Q(X,Y):= -3X+\frac{5Y^2-6Y+5}{1+Y}.
\end{align*} 
If $Y$ has a minimum at time $s_0$ with $Y(s_0) \geq 0$, then $Q(X(s_0),Y(s_0)) \leq 0$. 
If $Y$ has a maximum at time $s_0$ with $Y(s_0) \geq 0$, then $Q(X(s_0),Y(s_0)) \geq 0$.
\end{compactenum}
\end{lemma}
\begin{proof}
The first statement follows easily from the evolution equation for $Y$. To determine the nature of critical points we compute the second derivative of $Y$ at a critical point:
\begin{align*}
\frac{d^2}{ds^2}Y
=
\frac{d}{ds}
\left(
4Y-4Y^2-2YZ-2Z
\right)
=
-2(1+Y) \dot{Z}
.
\end{align*}
Hence, at a critical point of $Y$ with $Y\geq 0$ the second derivative of $Y$ has the opposite sign as the first derivative of $Z$.
We can use the equation $\dot{Y}=0$ to solve for $Z$ and obtain
\begin{align*}
\frac{d}{ds}\log(Z)
&=
5-3X-3Y-4Z
\\
&=
5-3X-3Y-8Y\frac{1-Y}{1+Y}
=
-3X+\frac{5Y^2-6Y+5}{1+Y}
.
\end{align*} 
We see that a critical point of $Y$ is a minimum
if $Q(X,Y) \leq 0$. The statement for maxima of $Y$ follows analogously.
\end{proof}

\subsection{Proofs of Theorems \ref{theorem1}, \ref{thmB} and \ref{thm-sing}}

\label{section-proofs}

\textit{Theorem \ref{theorem1}}.
We will now proof Theorem \ref{theorem1}. 
Set
\begin{align*}
\frak{X}_{\mathrm{alc}}
&
:=\{\mu\ |\ \Psi_{\mu}\ \text{is complete and}\ \lim_{s\rightarrow\infty} (X,Y,Z) = (1,1,0) \},
\\
\frak{X}_{\mathrm{ac}} 
&
:= \{\mu\ |\ \Psi_{\mu}\ \text{is complete and}\ \lim_{s\rightarrow \infty}(X,Y,Z)=(X_c,Y_c,Z_c)  \},
\\
\frak{X}_{\mathrm{inc}} 
&
:=  \{\mu\ |\ \Psi_{\mu}\ \text{is incomplete} \}.
\end{align*}

By Proposition \ref{prop-ALC crit point} $\Psi_{\mu}$ is ALC if $\mu\in\frak{X}_{\mathrm{alc}}$. By Proposition \ref{lemma - crit point cone is AC}
$\Psi_{\mu}$ is AC if $\mu\in\frak{X}_{\mathrm{ac}}$.
By Lemma \ref{bounds in xyz coords} and Proposition \ref{prop-crit point convergence} we know that the trajectory associated with the Spin(7)-structure $\Psi_{\mu}$ either hits the hypersurface $Y=0$, i.e. the Spin(7)-structure is incomplete and $\mu\in\mathfrak{X}_{\mathrm{inc}}$, or stays inside the cube 
\begin{align*}
\mathcal{W}=(0,1)\times(0,1)\times(0,5/4)
\end{align*}
for all times and converges to a fixed point. This fixed point must lie in the closure of $\mathcal{W}$ and therefore is one of those listed in Remark \ref{critical points}. Because $(0,0,0)$ is a source and $\Psi_{\mu}$ certainly does not lie in the relevant branches  \eqref{Y-axis} and \eqref{X-axis} of the  1-dimensional stable manifolds of the fixed points $(0,1,0)$ and $(1,0,0)$, respectively, we see that $(0,\infty)$ is the disjoint union of $\frak{X}_{\mathrm{alc}}, \frak{X}_{\mathrm{ac}}$ and $\frak{X}_{\mathrm{inc}}$. 

The key idea of the proof is to show that $\frak{X}_{\mathrm{alc}}$
and $\frak{X}_{\mathrm{inc}}$ are both open and non-empty. 
$\frak{X}_{\mathrm{alc}}$ is open because $\Psi_{\mu}$ depends continuously on $\mu$ and $(1,1,0)$ is a sink.
Suppose $\mu\in\frak{X}_{\mathrm{inc}}$ and thus $Y_{\mu}(s_0)=0$ for some $s_0$. Then at time $s_0$ we have $\dot{Y}=-2Z(s_0)<0$, i.e. the flow line actually exits the closure of $\mathcal{W}$, which again is an open condition due to the continuity of $\Psi_{\mu}$ in $\mu$.

To see that $\frak{X}_{\mathrm{alc}}$
and $\frak{X}_{\mathrm{inc}}$ are non-empty, we look at the limiting trajectories as $\mu\rightarrow 0$ and $\mu \rightarrow \infty$
and use that up to the limit the flow lines depend continuously on $\mu$.
As $\mu\rightarrow 0$, the limiting trajectory is \eqref{BS-XYZ} which corresponds to the Bryant--Salamon $\mathrm{G}_2$ holonomy metric on $\Lambda^2_{-}\CP{2}$ and connects the fixed points $(0,1,0)$ and $(1,1,0)$.
Because $(1,1,0)$ is a sink, for small $\mu > 0$ the trajectories associated with $\Psi_{\mu}$ also must converge to $(1,1,0)$. 
As $\mu \rightarrow \infty$ the limiting trajectory initially is contained in the face $X=0$ of the cube $\mathcal{W}$. This trajectory cannot converge to a fixed point because the condition $X=0$ is preserved, the branch \eqref{Y-axis} of the stable manifold of $(0,1,0)$ originates in $(0,0,0)$, and the only other fixed point with $X$-coordinate equal to zero is the source $(0,0,0)$. By Proposition \ref{prop-crit point convergence} it therefore must exit the closure of the cube $\mathcal{W}$ after some time. 
Similarly as in Lemma \ref{bounds in xyz coords}, this is only possible at the hypersurface $Y=0$. Because exiting the closure of $\mathcal{W}$ is an open condition, we have $\mu\in\mathfrak{X}_{\mathrm{inc}}$ if $\mu$ is sufficiently large. 

Because $(0,\infty)$ is connected, $\mathcal{X}_{\mathrm{ac}}$ is non-empty and by Lemmas \ref{comparison for Phu-mu} and \ref{uniqueness-simple} consists of exactly one parameter $\mathfrak{X}_{\mathrm{ac}}=\{ \mu_{\mathrm{ac}} \}$. 
Suppose that $\mu \in (0,\mu_{\mathrm{ac}})$.
We know that either $\mu\in\mathfrak{X}_{\mathrm{alc}}$ or $\mu\in\mathfrak{X}_{\mathrm{inc}}$. By Lemma \ref{comparison for Phu-mu} 
we have $Y_{\mu}(s) > Y_{\mu_{\mathrm{ac}}}(s)$ as long as $\Psi_{\mu}$ exists. Because $Y_{\mu_{\mathrm{ac}}}(s)\rightarrow Y_c$ as $s\rightarrow\infty$, $\Phi_{\mu}$ is complete by Lemma \ref{lemma-reduction to proj system} and thus
we get $\mu\in \mathfrak{X}_{\mathrm{alc}}$. 
Next suppose that $\mu > \mu_{\mathrm{ac}}$. With Lemma \ref{comparison for Phu-mu} we get $Y_{\mu}(s) < Y_{\mu_{\mathrm{ac}}}(s)$ as long as $\Psi_{\mu}$ exists. Because $Y_{\mu_{\mathrm{ac}}}$ converges to $Y_c < 1$   we have $\mu \notin \mathfrak{X}_{\mathrm{alc}}$. Hence we get $\mu\in\mathfrak{X}_{\mathrm{inc}}$. This finishes the proof of Theorem \ref{theorem1}. 

\vspace*{5mm}

\textit{Theorem \ref{thmB}}. The proof is analogous to the proof of Theorem \ref{theorem1}.
Moreover, formulas \eqref{Spin(7)-structure} and \eqref{short-distance-asy-exp-cp2}
show that $\Upsilon_{\tau}|_{\CP{2}}= e_{4356}$ (with the original use of the functions $a$ and $b$). This is the volume form of $\CP{2}$ with respect to the induced metric and the appropriate orientation. Therefore, the zero section is a Cayley submanifold with respect to $\Upsilon_{\tau}$ for all $\tau$. Because the $\CP{2}$ is a generator of $H_4(M_{\CP{2}})$
and its volume with respect to $\Upsilon_{\tau}$ is positive and independent of $\tau$,
the cohomology class of $\Upsilon_{\tau}$ is non-trivial and does not depend on $\tau$.

\vspace*{5mm}

\textit{Theorem \ref{thm-sing}}.
Now that we know how to prove Theorems \ref{theorem1} and \ref{thmB} it is easy to find among the 2-dimensional space of flow lines emanating from the source $(0,0,0)$ 1-parameter families with the same transition behaviour of the asymptotic geometries as the families $\Psi_{\mu}$ and $\Upsilon_{\tau}$. If we make sure that they initially move inside the cube $\mathcal{W}$, that pairs of family members can be compared as in Lemma \ref{comparison for Phu-mu}, and that the flow lines at one end of the interval limit to a flow line which converges to the fixed point $(1,1,0)$ and at the other end limit in a flow line which after some time enters the region $Y<0$, the proof of Theorem \ref{theorem1} carries through as before. One possible way to choose these 1-parameter families is to pick small enough positive constants $\varepsilon$ and $z_0$ such that
all flow lines through points on the line segments $(1-\kappa) (\epsilon,\epsilon,0)+\kappa(\epsilon,0,z), \kappa\in(0,1), z\in(0,z_0)$, originate in $(0,0,0)$. For fixed $z$ denote this 1-parameter family of Spin(7) holonomy metrics by $\Omega^{z}_{\kappa}$. This gives Theorem \ref{thm-sing}.

The families $\Omega^z_{\kappa}$ have been chosen in such a way that
as $\kappa \rightarrow 0$ they collapse to the explicit solution 
\eqref{diagonal},
which is a singular version of the Bryant-Salamon metric corresponding to the case $C=1$ in Remark \ref{rem-reduced system}.
The corresponding solution in $(a,b,c,f)$ coordinates with respect to the $s$-parameter is given by
\begin{align}
\label{singularity}
a(s) = b(s) = e^s, \quad c(s) = \sqrt{e^{2s}+e^{-2s}}, \quad f(s) =0.
\end{align}
We see that $c(s)$ blows up as $s\rightarrow -\infty$. 
Reidegeld \cite[Theorem 5.4.6]{reidegeld} shows that up to discrete symmetries and scale there can be no complete $\SU(3)\times\U(1)$-invariant with principal orbit $\SU(3)$-equivariantly diffeomorphic to $N(1,-1)$ except members of the families $\Psi_{\mu}$ and $\Upsilon_{\tau}$. Therefore, all members of the families $\Omega^z_{\kappa}$ must be singular on the end corresponding to $s\rightarrow -\infty$. Because for $\kappa>0$ we do not know explicit expressions for the solutions $\Omega^z_{\kappa}$, we can only speculate that as $s\rightarrow -\infty$ they behave similar to  \eqref{singularity}.

\subsection{Proof of Theorem \ref{Theorem 2}}

To prove Theorem \ref{Theorem 2} we use a more quantitative approach.
Lemma \ref{lemma-growth of Y} suggests that that we can study extrema of $Y$ by ignoring $Z$ and considering the projection of the trajectory to the $(X,Y)$-plane. More specifically, the curve $Q(X,Y)=0$ partitions the unit square in two disjoint regions such that $Y$ can have a minimum only in one of them and a maximum only in the other one.
To extract information on the asymptotic geometry of complete solutions we should compare the evolution of $Y$ with $Y_c$, the $Y$-coordinate of the fixed point corresponding to the conical solution. In fact, the intersection of the curves $Q(X,Y) = 0$ and $Y = Y_c$ is precisely the projection $(X_c, Y_c)$ of the cone $(X_c, Y_c, Z_c)$. 
We partition the unit square as 
\begin{align*}
\mathcal{D}_1
=
\{
(X,Y)\in[0,1]^2|\ Q(X,Y) > 0,\ Y > Y_c
\},
\\
\mathcal{D}_2
=
\{
(X,Y)\in[0,1]^2|\ Q(X,Y) \leq 0,\ Y > Y_c
\},
\\
\mathcal{D}_3
=
\{
(X,Y)\in[0,1]^2|\ Q(X,Y) > 0,\ Y \leq Y_c
\},
\\
\mathcal{D}_4
=
\{
(X,Y)\in[0,1]^2|\ Q(X,Y) \leq 0,\ Y \leq Y_c
\}
.
\end{align*}

Note that $Q(X,Y)=0$ is equivalent to
\begin{align}
X=\frac{1}{3}\frac{5Y^2-6Y+5}{1+Y}.
\label{equation for x,y on boundary}
\end{align}

\pgfplotsset{ticks=none}

\begin{center}

\begin{tikzpicture}

samples = 500

\pgfmathsetmacro{\XC}{(3/10)*(5-sqrt(5))}
\pgfmathsetmacro{\YC}{(3-sqrt(5))/2}
\pgfmathsetmacro{\BP}{(9-sqrt(41))/10}

\begin{axis}[axis equal image, tick style = {draw=none}, 
xmin = 0, xmax = 1, ymin = 0, ymax = 1, clip mode = individual]

\addplot[name path = xaxis, domain = 0:1, draw = none]{0};

\addplot[name path = upperxaxis, domain = 0:1, draw = none]{1};

\addplot[name path = doomed, thick, domain = 0:1] {\YC};

\addplot[name path = b2, domain = \BP:\YC, thick]((1/3)*(5*x*x-6*x+5)/(1+x),x);

\addplot[name path = b1, thick, domain = \BP:1]((1/3)*(5*x*x-6*x+5)/(1+x),max(x,\YC);

\addplot[gray, opacity = 0.1, domain = 0:\YC](\XC,x);

\addplot[fill = gray,opacity = 0.5]fill between[of = b1 and upperxaxis];

\addplot[fill = gray, opacity = 0.5]fill between[of = xaxis and doomed, soft clip first = {domain = 0:\XC}, soft clip second = {domain = {0:\XC}}];

\addplot[fill = gray, opacity = 0.5]fill between[of = xaxis and b2, soft clip first = {domain = \XC:1}];

\addplot[fill = black]fill between[of = b2 and doomed, soft clip second = {domain=\XC:1}];

\draw[] (axis cs: 0 , 1) node[circle , fill, inner sep = 2pt, label = {below right: $S^5$}]{};

\draw[] (axis cs: 1 , 1) node[circle, fill, inner sep = 2pt, label = {below left: ALC}]{};

\draw[] (axis cs: \XC, \YC) node[circle , fill, inner sep = 2pt, label = {below left: Cone}]{};

\draw[] (axis cs: 1, 0) node[circle, fill, inner sep = 2pt, label = above left: $\CP{2}$]{};

\node[] at (axis cs: 0.3,0.7) {$\mathcal{D}_1$};

\node[] at (axis cs: 0.85,0.7) {$\mathcal{D}_2$};

\node[] at (axis cs: 0.5,0.2) {$\mathcal{D}_3$};

\node[white] at (axis cs: 0.95, 0.34) {$\mathcal{D}_4$};

\end{axis}
\end{tikzpicture}

\end{center}

We can now reformulate Lemma \ref{lemma-growth of Y}
as
\begin{lemma}
\label{regions for extrema}
Let $(X, Y, Z)$ be a local solution contained inside the unit square. The function $Y(s)$ can have a minimum only in $\mathcal{D}_2 \cup \mathcal{D}_4$ and a maximum only in $\mathcal{D}_1 \cup \mathcal{D}_3$.
\end{lemma}

The key idea in the proof of Theorem \ref{Theorem 2} is that the region $\mathcal{D}_2$  is a trap if $\dot{Y}>0$, and the region $\mathcal{D}_3$ is a trap if $\dot{Y}<0$. The following Lemma makes this precise.

\begin{lemma}
\label{barrier lemma}
Let $(X(s), Y(s), Z(s))$ be a solution of the system \eqref{ODE system in good coord} satisfying the conditions of Lemma \ref{bounds in xyz coords}.
\begin{compactenum}[(i)]
\item
Assume that at some time we have $\dot{Y} \geq 0$ and the solution projects to the interior of the common boundary of $\mathcal{D}_1$ and $\mathcal{D}_2$. Then at this time we have 
\begin{align*}
\frac{d}{ds} Q(X,Y) < 0.
\end{align*}
In particular, if at some time the solution is in $\mathcal{D}_2$ while $\dot{Y}\geq 0$, from then onwards, it is trapped in $\mathcal{D}_2$ and $Y$ is monotone increasing.
\item 
Assume that at some time we have $\dot{Y} \leq 0$ and the solution projects to the interior of the common boundary of $\mathcal{D}_3$ and $\mathcal{D}_4$. Then at this time we have 
\begin{align*}
\frac{d}{ds} Q(X,Y) > 0.
\end{align*}
In particular, if at some time the solution is in $\mathcal{D}_3$ while $\dot{Y}\leq 0$, from then onwards as long as $Y \geq 0$, it is trapped in $\mathcal{D}_3$  and $Y$ is monotone decreasing.
\end{compactenum}
\end{lemma}

\begin{figure}
\centering
\includegraphics[scale=.4]{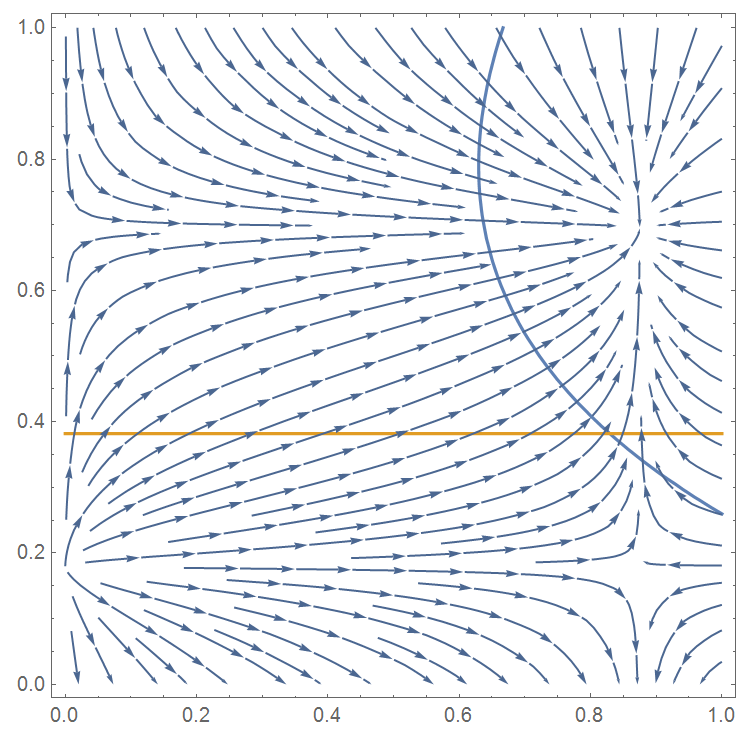}
\quad\quad
\includegraphics[scale=.4]{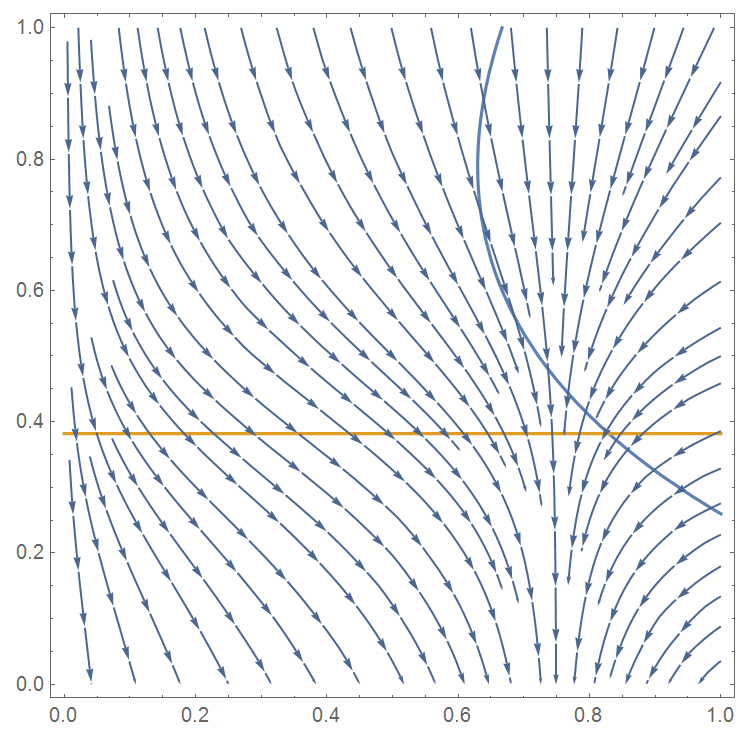}
\caption{To illustrate Lemma \ref{barrier lemma},
we have plotted the projection of the vector field to the $(X,Y)$-plane at the slices $Z=1/4$ and $Z=1/2$. At the common  boundary of $\mathcal{D}_1$ and $\mathcal{D}_2$, the vector field points into $\mathcal{D}_2$ if the $Y$-component is positive. At the common boundary of $\mathcal{D}_3$ and $\mathcal{D}_4$, the vector field points into $\mathcal{D}_3$ if the $Y$-component is negative.}
\end{figure}

\begin{proof}
\textbf{(i):}
We have 
\begin{align*}
\nabla Q(X,Y)
=
\left(
-3,\frac{5Y^2+10Y-11}{(1+Y)^2}
\right).
\end{align*}
Because $\dot{Y} \geq 0$, by Lemma \ref{lemma-growth of Y} (i) we have $Z \leq 2Y\frac{1-Y}{1+Y}$. Combining this with (\ref{equation for x,y on boundary})
we get
\begin{align}
\partial_X Q(X,Y) \dot{X}
&=
-6X(2-2X-Z)
\nonumber
\\
&
\leq
-2\frac{5Y^2-6Y+5}{1+Y}
\left(
2-\frac{2}{3}\frac{5Y^2-6Y+5}{1+Y}-2 Y \frac{1-Y}{1+Y}
\right)
\nonumber
\\
&=
\frac{8}{3}
\frac{(Y^2-3Y+1)(5Y^2-6Y+5)}{(1+Y)^2}.
\label{barrier estimate 1}
\end{align}
This function is negative if $Y> Y_c = \frac{3-\sqrt{5}}{2}$ and positive if $Y < Y_c = \frac{3-\sqrt{5}}{2}$.

Because $\partial_X Q(X,Y) \dot{X}$ is negative on the common boundary between $\mathcal{D}_1$ and $\mathcal{D}_2$ and $\dot{Y} \geq 0$ by assumption, we can assume $\partial_Y Q(X,Y)$ to be non-negative. This allows the estimate 
\begin{align}
\partial_Y Q(X,Y) \dot{Y}
&=
\partial_Y Q(X,Y)
(4Y-4Y^2-2YZ-Z)
\nonumber
\\
& 
\leq
\partial_Y Q(X,Y)4Y(1-Y)
=
\frac{5Y^2+10Y-11}{(1+Y)^2}
4Y(1-Y).
\label{barrier estimate 2}
\end{align}
Combining (\ref{barrier estimate 1}) and (\ref{barrier estimate 2}) we get
\begin{align*}
\frac{d}{ds} Q(X,Y)
&=
\nabla Q(X,Y)\cdot (\dot{X},\dot{Y})
\\
&
\leq
\frac{8}{3}
\frac{(Y^2-3Y+1)(5Y^2-6Y+5)}{(1+Y)^2}
+
\frac{5Y^2+10Y-11}{(1+Y)^2}
4Y(1-Y)
\\
&=
-\frac{4}{3}
\frac{5Y^4+57Y^3-119Y^2+75Y-10}{(1+Y)^2}
,
\end{align*}
which is negative for $Y > \frac{1}{5}$. 

By Lemma \ref{regions for extrema} the function $Y$ cannot have a maximum as long as the solution is in $\mathcal{D}_2$. 
Because the conditions $X <1$ and $Y < 1$ are preserved by Lemma \ref{bounds in xyz coords} and $\dot{Y} \geq 0$ as long as the solution is in $\mathcal{D}_2$, it can exit $\mathcal{D}_2$ only along the common boundary of $\mathcal{D}_1$ and $\mathcal{D}_2$. This was shown to be impossible.

\vspace{2mm}

\textbf{(ii):}
Because now $\dot{Y} \leq 0$, we get (\ref{barrier estimate 1}) with reversed inequality sign. 
This estimates 
$\partial_X Q(X,Y) \dot{X}$ from below by a function which is positive if $Y < Y_c = \frac{3-\sqrt{5}}{2}$. Furthermore $\partial_Y Q(X,Y)$ is negative if $Y < Y_c = \frac{3-\sqrt{5}}{2}$. Because we assume $\dot{Y} \leq 0$ we get 
\begin{align*}
\frac{d}{ds} Q(X,Y) =
\nabla Q(X,Y) \cdot (\dot{X},\dot{Y})=
\partial_X Q(X,Y) \dot{X}
+
\partial_Y Q(X,Y) \dot{Y}
>0.
\end{align*}
By Lemma \ref{regions for extrema} the function $Y$ cannot have a minimum in $\mathcal{D}_3$.
As $ 0 < X < 1$ is preserved by Lemma \ref{bounds in xyz coords} and $\dot{Y} \leq 0$ by assumption, the solution can exit 
$\mathcal{D}_3$ only at $Y = 0$ or on the common boundary with $\mathcal{D}_4$. The latter was shown to be impossible.
\end{proof}

\begin{proof}[Proof of Theorem \ref{Theorem 2}]
It is clear that $\lambda=0$ gives the Spin(7)-cone.
The short distance asymptotic expansion (\ref{short time asy exp cs sol}) for $\Psi^{\mathrm{cs}}_{\lambda}$ gives
\begin{align*}
Y(t)
&
\approx
Y_c(1-9.86 \lambda t^{\nu_2}+\mathcal{O}(t^{2\nu_2})),
\\
Q(X(t),Y(t))
&
\approx
14.41 \lambda t^{\nu_2}+\mathcal{O}(t^{2\nu_2}).
\end{align*}
If $\lambda < 0$, this implies that $\Psi^{\mathrm{cs}}_{\lambda}$
enters the region $\mathcal{D}_2$ with $\dot{Y}> 0$. By Lemma \ref{barrier lemma} (i) it follows that the solution is trapped in $\mathcal{D}_2$ for all times and $Y$ is monotone increasing. By Lemma \ref{lemma-reduction to proj system} it is forward complete and by Proposition \ref{prop-crit point convergence} it has to converge to a fixed point $(X_{\infty}, Y_{\infty}, Z_{\infty})$ which projects onto $\mathcal{D}_2$ with $Y_{\infty} > Y_c$. The only such fixed point is $(1,1,0)$ and hence $\Psi^{\mathrm{cs}}_{\lambda}$ is ALC by Proposition \ref{prop-ALC crit point}. 

If $\lambda > 0$ the solution enters $\mathcal{D}_3$ with $\dot{Y}<0$.  By Lemma \ref{barrier lemma} (ii) it is trapped there as long as $Y \geq 0$ and $Y$ is monotone decreasing. 
If the solution was forward complete, then by Proposition \ref{prop-crit point convergence} it would converge to some fixed point $(X_{\infty}, Y_{\infty}, Z_{\infty})$ with $Y_{\infty} < Y_c$. But $(0,0,0)$ is a source and $\Psi^{\mathrm{cs}}_{\lambda}$ certainly does not sweep out the 1-dimensional stable manifold of $(1,0,0)$ given by the solution \eqref{X-axis}. Therefore, $\Psi^{\mathrm{cs}}_{\lambda}$ cannot extend to a forward complete metric if $\lambda>0$.
\end{proof}

\begin{figure}[h]
\includegraphics[scale=.4]{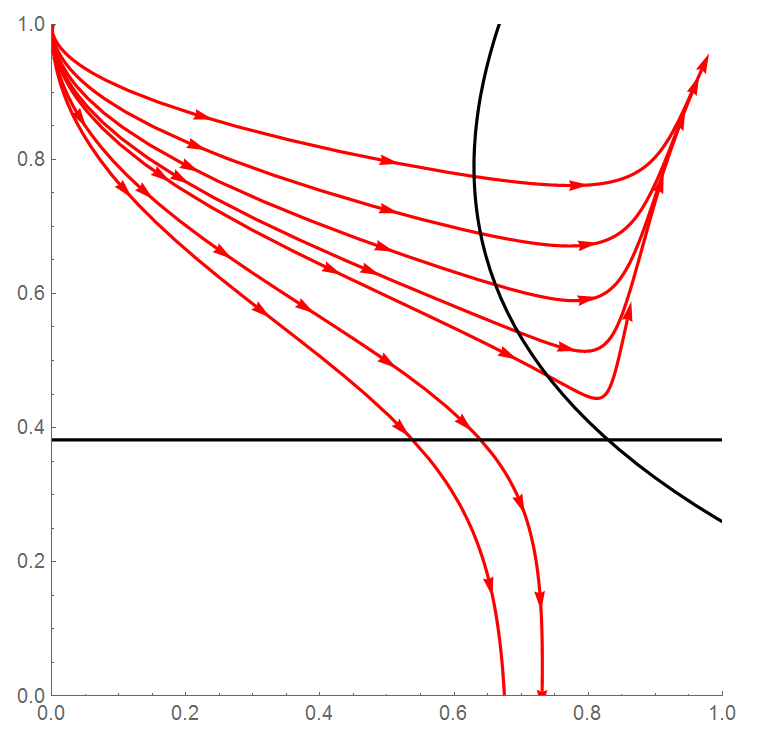}
\quad\quad
\includegraphics[scale=.4]{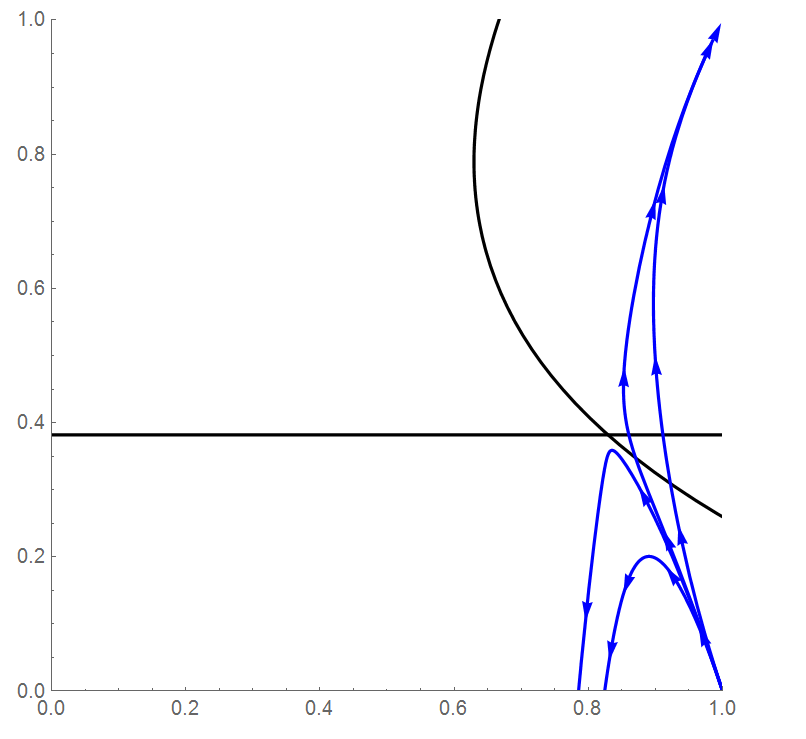}
\caption{Left: $\Psi_{\mu}$ is ALC if and only if the $Y$-coordinate of the associated flow line has a minimum, which must happen in $\mathcal{D}_2$. Right: $\Upsilon_{\tau}$ is incomplete if and only if the $Y$-coordinate of the associated flow line has a maximum, which must happen in $\mathcal{D}_3$.}
\label{2Dplot}
\end{figure}

Lemma \ref{barrier lemma} and a similar trapping argument can also be used to give more quantitative proofs of Theorems \ref{theorem1} and \ref{thmB} as in the author's PhD-thesis \cite{myPhD}. 
E.g., for the family $\Psi_{\mu}$ we can show that $\Psi_{\mu}$
is ALC if and only if the function $Y_{\mu}(s)$ has a minimum at some time. That this condition is necessary is clear: $Y_{\mu}(s)$ starts out at the value $1$ and is initially decreasing. Because the $Y$-coordinate of the ALC end $(1,1,0)$ is equal to $1$, $Y_{\mu}(s)$ must have a minimum at some time. With the help of Lemma \ref{barrier lemma} one can show that $Y_{\mu}(s)$ can only have a minimum while the solution is in the region $\mathcal{D}_2$. But then it is trapped there with $Y_{\mu}(s)$ monotone increasing and therefore it must converge to $(1,1,0)$. While this approach is more intricate, it has the added benefit that it gives us a good idea about the actual shape of the associated trajectories, which we illustrate in figure \ref{2Dplot}.

\bibliographystyle{amsalpha}

\bibliography{bibliography}

\providecommand{\bysame}{\leavevmode\hbox to3em{\hrulefill}\thinspace}
\providecommand{\MR}{\relax\ifhmode\unskip\space\fi MR }
\providecommand{\MRhref}[2]{%
  \href{http://www.ams.org/mathscinet-getitem?mr=#1}{#2}
}
\providecommand{\href}[2]{#2}
\begin{thebibliography}{CGLP02b}

\bibitem[AH88]{AH}
M.~F. Atiyah and N.~Hitchin, \emph{The {G}eometry and {D}ynamics of {M}agnetic
  {M}onopoles}, Princeton University Press, 1988.

\bibitem[Baz07]{bazaikin1}
Y.~V. Bazaikin, \emph{On the new examples of complete noncompact
  {S}pin(7)-holonomy metrics}, Siberian Mathematical Journal \textbf{48}
  (2007), no.~1, 8--25.

\bibitem[Baz08]{bazaikin2}
\bysame, \emph{Noncompact {R}iemannian {S}paces with the {H}olonomy {G}roup
  {S}pin(7) and 3-{S}asakian {M}anifolds}, Proceedings of the Steklov Institute
  of Mathematics \textbf{263} (2008), no.~1, 2--12.

\bibitem[Bry87]{Bryant}
R.~L. Bryant, \emph{Metrics with exceptional holonomy}, Annals of Mathematics
  \textbf{126} (1987), 525--576.

\bibitem[BS89]{BS}
R.~L. Bryant and S.M. Salamon, \emph{On the construction of some complete
  metrics with exceptional holonomy}, Duke Mathematical Journal \textbf{58}
  (1989), no.~3, 829--850.

\bibitem[Cal79]{Calabi}
E.~Calabi, \emph{M\'etriques k\"ahl\'eriennes et fibr\'es holomorphes}, Annales
  scientifiques de l'\'Ecole Normale Sup\'erieure \textbf{12} (1979), no.~2,
  269--294.

\bibitem[Car81]{Carr}
J.~Carr, \emph{Applications of {C}entre {M}anifold {T}heory}, Applied
  Mathematical Sciences, vol.~35, Springer, 1981.

\bibitem[CGLP02a]{CGLP2}
M.~Cveti\v{c}, G.W. Gibbons, H.~L{\"u}, and C.N. Pope, \emph{Cohomogeneity one
  manifolds of {S}pin(7) and $\text{G}_2$ holonomy}, Phys. Rev. D \textbf{65}
  (2002), 106004.

\bibitem[CGLP02b]{CGLP1}
\bysame, \emph{New complete non-compact {S}pin(7) manifolds}, Nuclear Physics B
  \textbf{620} (2002), 29 -- 54.

\bibitem[EW00]{eschenburg-wang}
J.-H. Eschenburg and M.Y. Wang, \emph{Initial value problems for cohomogeneity
  one {E}instein metrics}, The Journal of Geometric Analysis \textbf{10}
  (2000), no.~1, 109--137.

\bibitem[FHN18]{FHN2}
L.~Foscolo, M.~Haskins, and J.~Nordstr{\"o}m, \emph{Infinitely many new
  families of complete cohomogeneity one $\mathrm{G}_2$-manifolds:
  $\mathrm{G}_2$-analogues of the {T}aub-{NUT} and {E}guchi-{H}anson spaces},
  2018, to appear in the Journal of the European Mathematical Society.

\bibitem[FKMS97]{NP-G2}
Th. Friedrich, I.~Kath, A.~Moroianu, and U.~Semmelmann, \emph{On nearly
  parallel {G}2-structures}, Journal of Geometry and Physics \textbf{23}
  (1997), no.~3, 259 -- 286.

\bibitem[Fos19]{Foscolo}
L.~Foscolo, \emph{Complete non-compact {S}pin(7) manifolds from self dual
  {E}instein 4-orbifolds}, 2019, to appear in Geometry {\&} Topology.

\bibitem[GS02]{Gukov-Sparks}
S.~Gukov and J.~Sparks, \emph{M-theory on {S}pin(7) manifolds}, Nuclear Physics
  B \textbf{625} (2002), no.~1-2, 3--69.

\bibitem[GST03]{Gukov-Sparks-Tong}
S.~Gukov, J.~Sparks, and D.~Tong, \emph{Conifold {T}ransitions and
  {F}ive-{B}rane {C}ondensation in {M}-theory on {S}pin(7) {M}anifolds},
  Classical and Quantum Gravity \textbf{20} (2003), no.~4, 665–705.

\bibitem[Joy00]{big-joyce}
D.~D. Joyce, \emph{{C}ompact {M}anifolds with {S}pecial {H}olonomy}, Oxford
  Mathematical Monographs, Oxford University Press, 2000.

\bibitem[Joy07]{small-joyce}
\bysame, \emph{{R}iemannian {H}olonomy {G}roups and {C}alibrated {G}eometry},
  Oxford Graduate Texts in Mathematics, Oxford University Press, 2007.

\bibitem[KY02]{Kanno-Yasui}
H.~Kanno and Y.~Yasui, \emph{On {S}pin(7) holonomy metric based on
  $\mathrm{SU}(3)/\mathrm{U}(1)$}, Journal of Geometry and Physics \textbf{43}
  (2002), no.~4, 310--326.

\bibitem[Leh20]{myPhD}
F.~Lehmann, \emph{Families of complete non-compact spin(7) holonomy manifolds},
  Ph.D. thesis, University College London, 2020.

\bibitem[Mos57]{Mostert}
P.~S. Mostert, \emph{On a compact {L}ie group acting on a {M}anifold}, Annals
  of Mathematics \textbf{65} (1957), no.~3, 447--455.

\bibitem[MS10]{NK}
A.~Moroianu and U.~Semmelmann, \emph{The {H}ermitian {L}aplace operator on
  nearly {K}{\"a}hler manifolds}, Communications in Mathematical Physics
  \textbf{294} (2010), no.~1, 251--272.

\bibitem[Per96]{Perko}
L.~Perko, \emph{Differential {E}quations and {D}ynamical {S}ystems}, {S}econd
  ed., Texts in Applied Mathematics, vol.~7, Springer, 1996.

\bibitem[Pic28]{Picard}
{\'E}.~Picard, \emph{Traite d'{A}nalyse}, {T}hird ed., vol.~3,
  Gauthier-Villars, 1928.

\bibitem[Rei08]{reidegeld}
F.~Reidegeld, \emph{{S}pin(7)-manifolds of cohomogeneity one}, 2008,
  PhD-thesis.

\bibitem[Rei10]{Reidegeld-paper}
\bysame, \emph{Exceptional holonomy and {E}instein metrics constructed from
  {A}loff-{W}allach spaces}, Proceedings of the London Mathematical Society
  \textbf{102} (2010), no.~6.

\bibitem[Sal89]{Salamon-book}
S.~Salamon, \emph{Riemannian geometry and holonomy groups}, Pitman Research
  Notes in Mathematics Series, vol. 201, Longman Scientific \& Technical,
  Harlow, 1989.

\bibitem[VZ18]{Verdiani-Ziller}
L.~Verdiani and W.~Ziller, \emph{Smoothness {C}onditions in {C}ohomogeneity one
  manifolds}, 2018, arXiv:1804.04680.

\end{thebibliography}

\end{document}